\documentclass[11pt]{amsart}
\usepackage{cite}
\usepackage{pinlabel}
\usepackage{amsmath} 
\usepackage{amsthm} 

\usepackage{graphicx} 
\setkeys{Gin}{keepaspectratio}
\usepackage{hyperref}
\usepackage{url}
\usepackage{bm}
\usepackage{mathabx}
\usepackage[left=1.25in,top=1in,right=1.25in,bottom=1in,head=.1in]{geometry}
\usepackage{xcolor}
\usepackage{tikz}
\usetikzlibrary{cd}
\usepackage{adjustbox}
\usepackage[normalem]{ulem} % 
\usepackage{enumitem}

\usepackage{verbatim}
\usepackage{xspace}
\newcommand{\items}{\begin{itemize}[leftmargin=25pt,rightmargin=15pt]
  \setlength\itemsep{2pt}}
\newcommand{\stopitems}{\end{itemize}}

\usepackage[textsize=scriptsize]{todonotes}

\usepackage{leftidx}% http://ctan.org/pkg/leftidx
\usepackage{fancyhdr}
\newtheorem{theorem}{Theorem}[section] % numbered theorems, lemmas, etc
\newtheorem*{theorem*}{Theorem}
\newtheorem{lemma}[theorem]{Lemma}

\newtheorem*{conjecture*}{Conjecture}
\newtheorem{question}[theorem]{Question}
\newtheorem*{question*}{Question}
\newtheorem*{lemma*}{Lemma}
\newtheorem{proposition}[theorem]{Proposition}

\newtheorem*{corollary*}{Corollary}
\theoremstyle{definition}
\newtheorem{definition}[theorem]{Definition}
\newtheorem{remark}[theorem]{Remark}

\newtheorem*{remark*}{Remark}

\newtheorem{example}[theorem]{Example}

\newtheorem*{example*}{Example}
\newtheorem*{remarks*}{Remarks}
\newtheorem*{addenda*}{Addenda}
\newtheorem*{construction*}{Construction}

\newcommand{\cc}{{\bf c}}  %%% replace with line below if we prefer eta
\newcommand{\nnu}{\boldsymbol{\nu}}  %%% replace with line below if we prefer plain \nu

\newcommand{\bM}{{\bf M}}

\newcommand{\bX}{{\bf X}}
\newcommand{\bY}{{\bf Y}}
\newcommand{\bZ}{{\bf Z}}

\newcommand{\bff}{\boldsymbol{f}}

\newcommand{\sss}{S^2\mkern-1.5mu \times \mkern-1.25mu S^2} % \ss is German double-s

\newcommand{\mm}{\bM} % notation for specific manifold M = CP^ \# 2 \cptwobar
\newcommand{\xx}{\bX} % notation for specific manifold X with D \neq 0
\newcommand{\yy}{\bY} % notation for specific manifold Y with D = 0
\newcommand{\zz}{\bZ} % notation for specific manifolds Z whose higher D's we calculate

\DeclareMathOperator{\sub}{Sub}

\DeclareMathOperator{\diff}{Diff}

\DeclareMathOperator{\emb}{emb}

\DeclareMathOperator{\homeo}{{Homeo}^0}
\DeclareMathOperator{\Top}{Top}  %%% or if we prefer change to C^0 superscript
 
\newcommand{\R}{\mathbb R}
\newcommand{\Z}{\mathbb Z}

\newcommand{\bc}{\mathbb C}

\newcommand\Chat{\widehat{C}}
\DeclareMathOperator{\SO}{SO}

\DeclareMathOperator{\U}{U}

\newcommand{\bone}{{\bf 1}}

\newcommand{\cs}{\mathbin{\#}}
\newcommand{\cpone}{\bc P^1}

\newcommand{\cptwo}{\bc P^2}

\newcommand{\cptwobar}{\smash{\overline{\bc P}^2}}

\newcommand{\spinc}{\operatorname{spin^c}}

\newcommand{\unred}[1]{ \ignorespaces}  %% to omit that red text
\newcommand{\Purple}[1]{{\color{purple}#1}}
\newcommand{\Blue}[1]{{\color{blue}#1}}
\newcommand{\Red}[1]{{\color{red}#1}}
\newcommand{\sw}{\operatorname{SW}}

\newcommand{\Bigwedge}{\mathord{\adjustbox{valign=B,totalheight=.6\baselineskip}{$\bigwedge$}}}

\begin{document}

% \title[short text for running head]{full title}
\title{Exotic families of embeddings}
% \\ \today}

%    Only \author and \address are required; other information is
%    optional.  Remove any unused author tags.

%    author one information
% \author[short version for running head]{name for top of paper}
\author{}
\address{}
\curraddr{}
\email{}
\thanks{}

%    author two information
\title{Exotic families of embeddings}
\author[Dave Auckly]{Dave Auckly}
\address{Department of Mathematics\newline\indent Kansas State University\newline\indent  Manhattan,
Kansas 66506}
\email{dav@math.ksu.edu}
\author[Daniel Ruberman]{Daniel Ruberman}
\address{Department of Mathematics, MS 050\newline\indent Brandeis
University \newline\indent Waltham, MA 02454}
\email{ruberman@brandeis.edu}
\thanks{Both authors were partially supported by NSF Grant DMS-1928930 while they were in residence at the Simons Laufer Mathematical Sciences Institute (formerly known as MSRI). We thank the institute for its hospitality. Further progress was made at our AIM SQuaRE, \emph{Family gauge theory and embedding spaces}, and we thank AIM for the congenial working environment. The first author was partially supported by Simons Foundation grant 585139 and NSF grant DMS-1952755. The second author was partially supported by NSF Grant DMS-1811111 and NSF grant DMS-1952790. \\
Math.~Subj.~Class.~2010: 57M25 (primary), 57Q60 (secondary).}
%\thanks{}

%    The 2010 edition of the Mathematics Subject Classification is
%    now available.  If you are citing a classification from the
%    new scheme, use the following input coding instead.
%\subjclass[2010]{Primary }

%\date{\today}

\begin{abstract} We construct a number of topologically trivial but smoothly non-trivial families of embeddings of $3$-manifolds in $4$-manifolds. These include embeddings of homology spheres in $S^4$ that are not isotopic but have diffeomorphic complements, and families (parameterized by high-dimensional spheres) of embeddings of any $3$-manifold that embeds in a blown-up $K3$ surface. In each case, the families are constructed so as to be topologically trivial in an appropriate sense. We also illustrate a general technique for converting a non-trivial family of embeddings into a non-trivial family of submanifolds. 
\end{abstract}
\dedicatory{Dedicated to Tom Mrowka on the occasion of his $61^{st}$ birthday}
\maketitle

%    Text of article.

\section{Introduction}
Starting
with the work of Donaldson~\cite{donaldson:polynomial} and Freedman~\cite{freedman:simply-connected}, many simply connected $4$-manifolds have been shown to admit exotic smooth structures. It is natural to look for other exotic phenomena related to $4$-manifolds, including exotic embeddings of surfaces, exotic embeddings of $3$-manifolds, exotic diffeomorphisms, and exotic families of such embeddings and diffeomorphisms.  
In this paper, we establish two theorems about exotic embeddings of $3$-manifolds in $4$-manifolds. The first gives infinitely many exotic embeddings of certain homology $3$-spheres in the $4$-sphere. The second gives exotic families, parameterized by higher dimensional spheres, of many $3$-manifolds in larger $4$-manifolds. 

\begin{theorem}\label{T:A}
There exist an infinite collection of integer homology $3$-spheres, $Y_n$, each with an infinite number of smooth embeddings
$j_{p,n}\colon Y_n \to S^4$ so that
\begin{enumerate}
\item $S^4\setminus j_{p,n}(Y_n) \cong \text{\rm int}(X_+(n)) \perp\!\!\perp \text{\rm int}(X_-(n))$,
\item $X_+(n) \cong - X_-(n)$,
\item $j_{p,n}$ is topologically ambient isotopic to $j_{q,n}$,
\item If $\psi\colon S^4\to S^4$ is a diffeomorphism satisfying $\psi\circ j_{p,n} = j_{q,n}$, then $p = q$.
\end{enumerate}
\end{theorem}
\begin{remark}
A recent MathOverflow~\cite{klug:MO} posting by Michael Klug asked if there were non-isotopic embeddings of $3$-manifolds into $S^4$ with diffeomorphic complements. A number of people noticed that the existence of corks with doubles and twisted doubles diffeomorphic to $S^4$ allows one to conclude that the answer to the question is yes. Ian Agol posted this as a reply on MathOverflow. The above theorem is our version, based on the infinite-order corks discovered by Gompf~\cite{gompf:infinite-cork} to give infinitely many non-isotopic embeddings.
\end{remark}

We then turn to the construction and detection of exotic families of embeddings. Here we illustrate a general principle by establishing the following result about smoothly knotted families of codimension one embeddings. Our convention is to denote by $\emb(Y,Z)$ the smooth embeddings of $Y$ in $Z$, and by $\emb^{\Top}(Y,Z)$ the space of locally flat topological embeddings.  Note that these are spaces of \emph{maps}, rather than submanifolds; the space $\sub(Y,Z)$ of submanifolds is the quotient of $\emb(Y,Z)$ by the diffeomorphism group of $Y$ (with similar notation for locally flat topological submanifolds.) With some additional hypotheses, the same results hold for families of submanifolds.  We address the case of submanifolds in the final section of the paper.

To state the main theorem about families of embeddings, we use the notation $E(2)$ for the $K3$ surface, and let $T$ be a smooth fiber in a particular elliptic fibration. 
\begin{theorem}\label{T:B}
Let $Y$ be a $3$-manifold that embeds smoothly in $E(2) \cs \cptwobar$ disjointly from $T$. For any $r>0$, {$k\geq 0$} 
there are smooth simply connected $4$-manifolds $Z^k_r$, together with  spherical classes $J_p^{j+1}\in H_{j+1}(\emb(Y,Z^k_r))$
for $j\le k$ and $j \equiv k, \ (\text{mod} \ 2)$ and $p = 1,\cdots,r$ that generate a rank $r$ summand of
\[
\ker\left[H_{j+1}(\emb(Y,Z^k_r))\to H_{j+1}(\emb^{\Top}(Y,Z^k_r))\right].
\]
Furthermore, when $Y$ is a homology sphere, the classes $J_p^{j+1}$ are 
trivial in $\pi_{j+1}(\emb^{\Top}(Y,Z^k_r))$, and trivial in 
$\pi_{j+1}(\emb(Y,Z^k_r\cs\sss))$.
\end{theorem}
As the reader will see, the constructions that go into the proof are very flexible. Hence the argument would work for many other $3$-manifolds, including those that embed disjoint from a nucleus in an arbitrary blowup of the elliptic surface $E(n)$. The specific choices embodied in Theorem~\ref{T:B} are described in Section~\ref{S:blocks}.

To prove this theorem, we need a method to construct examples, an invariant to distinguish the smooth families, and a method to compute the invariant. Defining and computing invariants builds on the second author's earlier papers~\cite{ruberman:isotopy,ruberman:polyisotopy,ruberman:swpos} on exotic diffeomorphisms. These defined invariants using both Donaldson theory and Seiberg-Witten theory and used them to show the existence of non-isotopic diffeomorphisms and non-isotopic metrics of positive scalar curvature. The story for families of diffeomorphisms and embeddings is much the same. It was easier to see how to define and compute invariants for families using Donaldson theory. These invariants depend upon a $\U(2)$ lift of an $\SO(3)$ bundle with suitable first Pontryagin class, and a homology orientation. There are only a finite number of possibilities of this data. This leads~\cite{auckly-ruberman:diffym} to results demonstrating large rank summands in the kernel for spaces of diffeomorphisms and embeddings of $S^3$.
{Many related works have appeared in the last few years, including~\cite{baraglia:k3,baraglia-konno:nielsen,botvinnik-watanabe:families,budney-gabai:3-balls,iida-konno-mukherjee-taniguchi:boundary,konno:classes,konno-mukherjee-taniguchi:codim1,kronheimer-mrowka:dehntwist,lin:dehn,lin:loops,lin-mukherjee:surfaces, nakamura:diffeomorphisms,nakamura:Z+Z,smirnov:loops,watanabe:S4,watanabe:theta}.}

After this paper was written, we have worked out how to define and compute invariants based on the Seiberg-Witten equations~\cite{auckly-ruberman:diffsw}. Here the data includes the choice of a $\spinc$ structure, of which there are an infinite number of suitable choices. The principle of the definition of the invariants is the same as in~\cite{ruberman:isotopy,konno:classes} but the calculation for higher dimensional families is new. This calculation leads to a strengthening of the result to conclude that the kernels of the maps on homology or homotopy contain infinite rank summands. It will gives similar statements about families of positive scalar curvature metrics. These stronger results derived via the using the Seiberg-Witten equations are not stated here, but will appear in a paper~\cite{auckly-ruberman:diffsw} in preparation. The paper~\cite{auckly-ruberman:diffym} also includes analogous results about families of embedded spheres of self-intersection one. Drouin has extended these results to analogous results about families of spheres of arbitrary self-intersection~\cite{drouin:embeddings}. The first author~\cite{auckly:subtle} has given many new examples of exotically knotted surfaces, even ones that remain exotic after stabilization an arbitrary number of times. Work in progress by the authors along with Konno, Mukherjee, and Taniguchi will further extend results about families of embedded surfaces to all genus and self-intersection numbers. 
\begin{remark}\label{R:cork}
In discussions about our work with various colleagues, it has been suggested that there should be a sort of `cork theorem' for diffeomorphisms and embeddings, possibly even a parameterized version. Recall that the cork theorem~\cite{curtis-freedman-hsiang-stong,matveyev:h-cobordism} says that two homotopy equivalent closed simply connected manifolds are related by removing a contractible piece (the cork) and regluing by an automorphism of the boundary (the cork twist). One might doing this for exotic pairs of embeddings, or showing that exotic pairs of diffeomorphisms also localize on some simple pieces of the underlying manifold. With this in mind, we point out that all of the exotic behavior in this paper and in~\cite{auckly-ruberman:diffym} does stem from a single cork twist, in fact Akbulut's original cork twist from~\cite{akbulut:contractible}. 

In contrast, the work of Konno-Mukherjee-Taniguchi~\cite{konno-mukherjee-taniguchi:codim1} shows that some exotic (up to isotopy) codimension-one embedding behavior persists after an unlimited number of stabilizations. Since any cork twist extends over a bounding manifold after some number of stabilizations, this suggests that there may not be any such a cork theorem for embeddings. 
This seems very intriguing as a topic for future research in this direction.
\end{remark}
\noindent
{\bf Conventions:} We will assume that all $4$-manifolds are smooth, and oriented and that all $3$-manifolds are oriented.  In addition, when considering gauge theoretic invariants of a manifold $X$, we will assume a specific choice of homology orientation, i.e., an orientation of $H^2_+(X;\R)$.  We take homology and cohomology with integer coefficients unless stated otherwise. We will generally assume that $4$-manifolds are simply connected, although many of our results will hold for more general classes of manifolds. A pair of manifolds $X$, $Y$ is \emph{exotic} if they are homeomorphic but not diffeomorphic; a collection $\{X_p\}$ is exotic if each pair $X_i$, $X_j$ is exotic (for $i \neq j$).\\[2ex]
{\bf Acknowledgements:} These results were inspired in part by conversations at MSRI with Hokuto Konno, Anubhav Mukherjee, and Masaki Taniguchi. We thank Chuck Livingston and Jeff Meier for some helpful correspondence about doubly slice knots, and the referee for a careful reading of our original submission. 

\section{Embeddings into \texorpdfstring{$S^4$}{S4w}}

Theorem~\ref{T:A} follows from the existence of chiral infinite order corks with trivial twisted doubles. The existence of these corks
follows in turn from the work of Gompf and Tange 
\cite{gompf:infinite-cork,gompf:infinite-cork2,tange}; see also~\cite{akbulut:infinite-corks} . This is the content of the following proposition.

\begin{proposition}\label{P:cork}
There are oriented, contractible, compact, smooth manifolds $X_+(n)$ together with orientation-preserving self-diffeomorphisms of the boundary, \hfill\newline$f(n)\colon \partial X_+(n) \to \partial X_+(n)$, and with embeddings $i(n)\colon X_+(n) \to E(2)\cs 2\cptwobar$ so that
\begin{enumerate}
\item If \begin{small}$\left(E(2)\cs 2\cptwobar \setminus \text{\rm int}(X_+(n))\right)\cup_{f(n)^p} X_+(n) \cong \left(E(2)\cs 2\cptwobar \setminus \text{\rm int}(X_+(n))\right)\cup_{f(n)^q} X_+(n)$\end{small}, then $p = q$. 
\item There are diffeomorphisms $\Phi_{p,n}\colon X_+(n)\cup_{f(n)^p} -X_+(n) \to S^4$,
\item There is a homomorphism  $\delta\colon  \pi_0(\diff(Y_n))\to\{\pm 1\}$ so that
\[
\alpha f(n)\alpha^{-1} = f(n)^{\delta(\alpha)}.
\]
\end{enumerate}
\end{proposition}

\begin{proof}[Proof of Proposition~\ref{P:cork}]
We consider a subfamily of the infinite order corks defined by Gompf \cite{gompf:infinite-cork,gompf:infinite-cork2}
set $X_+(n) = C(2,1;-2n)$ for $n>0$. 
Figure~\ref{X1} is a special case of \cite[Figure 3]{gompf:infinite-cork2} with the addition of the result after the zero-framed $2$-handle is canceled with one of the $1$-handles. 
\begin{figure}[ht]
\labellist
\pinlabel {$-2n$} [ ] at 161 190
\pinlabel {$0$} [ ] at 156 135
\pinlabel {$-2n$} [ ] at 161 65
\endlabellist
\centering
\includegraphics[scale = 1]{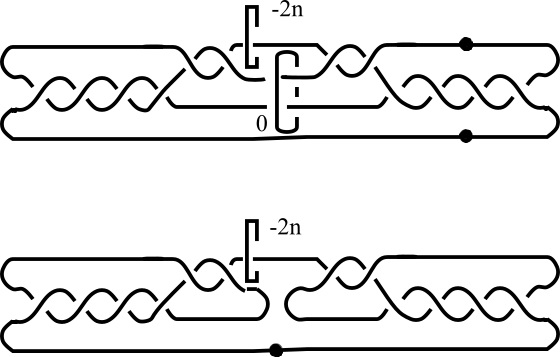}
  \caption{Infinite order corks $X_+(n)$}\label{X1}.
\end{figure}Figure~\ref{Y1} displays the boundary $Y_n$ of $X_+(n)$. It was obtained by a dot-zero swap followed by a Rolfsen twist. We see that this is just $1/2n$ Dehn surgery on a knot $K$, where $K$ is the sum of the stevedore knot and its reversed mirror, or more briefly, $Y_n = S^3_{1/2n}(K)$. This picture also includes a pair of incompressible tori in $Y_n$.
\begin{figure}[ht]
\labellist
\pinlabel {$\frac{1}{2n}$} [ ] at 195 85
\endlabellist
\centering
\includegraphics[scale = 1]{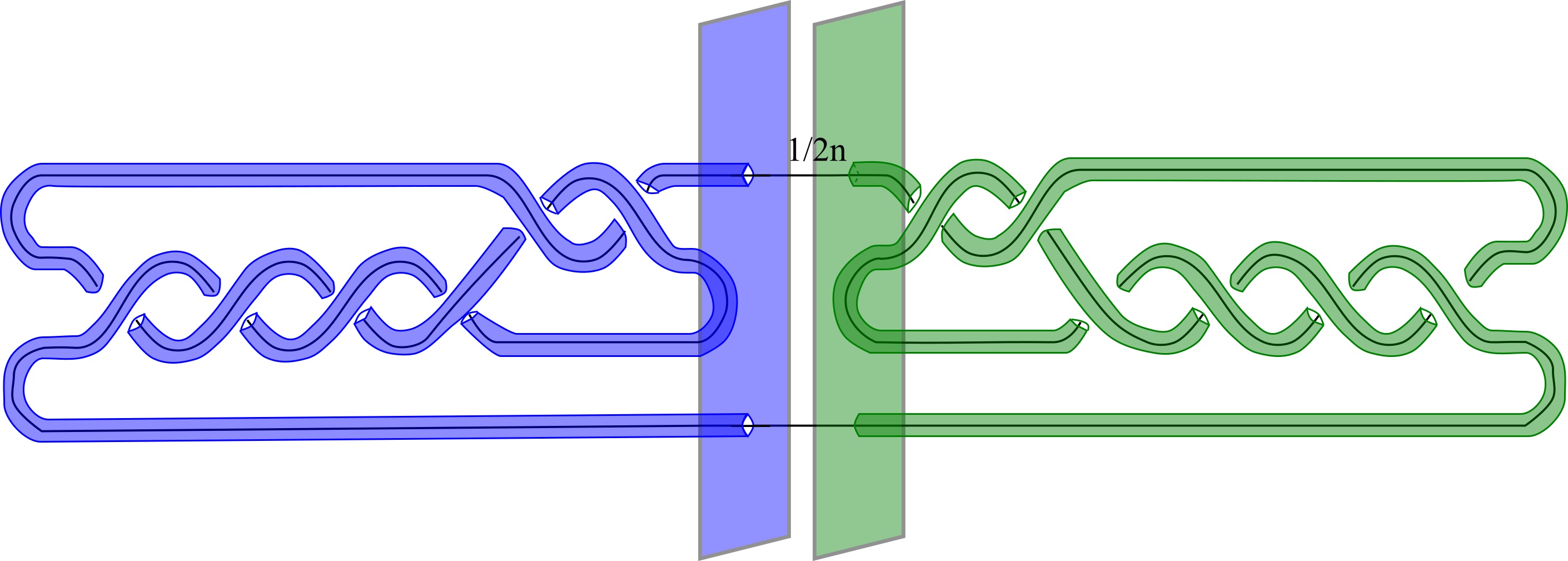}
  \caption{The homology spheres $Y_n$}\label{Y1}.
\end{figure}
The torus on the left of this figure separates the manifold $Y_n$ into two components. The  component on the left is the (mirror) of the exterior of the stevedore knot. Put coordinates on this torus as $(x,y)\in \R^2/\Z^2$ where $\R/\Z \times \text{point}$ is a longitude and $\text{point}\times\R/\Z$ is a meridian. Coordinates for a tubular neighborhood of the torus will then be $(x,y,t)$ and the diffeomorphism $f(n)$ is given by $f(n)(x,y,t) = (x+ t,y,t)$ in the neighborhood and the identity elsewhere. This map is called a torus twist; it is the product of a Dehn twist with the identity on $S^1$.

We can now prove the items in the proposition. The first two are just quoted results.
\begin{enumerate}
\item This is \cite[Theorem 3.5]{gompf:infinite-cork2},
\item This is  \cite[Theorem 2]{tange},
\item This follows quickly from the JSJ decomposition \cite{jaco-shalen:seifert,johannson:book}, as we now describe in more detail.
\end{enumerate}
The manifold $Y_n$ may be described via knot splicing, 
leading to its complete JSJ decomposition as displayed in Figure~\ref{Y1}. In the knot splice construction one identifies the neighborhood of a knot in one manifold with a neighborhood of a knot in a second manifold. One method to see a surgery description of knot splice is to take the product of an interval with the union of the original manifolds. Attaching one end of a $1$-handle to a neighborhood of a point of the knot in the first manifold and the other end to a neighborhood of a point produces a manifold with a boundary component that is obtained by just identifying these neighborhoods. Attaching a $2$-handle to the remainder of the knots and a pair of arcs along the $1$-handle completes the splice. Figure~\ref{Y2} displays $Y_n$ constructed in this way.  See also Budney's paper~\cite{budney:jsj} for other geometric calculations of this sort.

The knot complements arising in the splicing, going left to right in Figure~\ref{Y2}, are: the exterior of the mirror of the stevedore, $T^2\times I$, the Seifert fiber space with base orbifold having underlying surface a cylinder and one singular point of type $(-2n,1)$, another copy of $T^2\times I$, and  the exterior of the stevedore. The copies of $T^2\times I$ are not needed in the JSJ decomposition, but they are included here as a convenient way to describe the gluing via handle calculus. Since the exterior of the stevedore knot admits a complete hyperbolic structure there are no more incompressible tori that are not boundary parallel.
\begin{figure}[ht]
\labellist
\pinlabel {$-2n$} [ ] at 201 63
\pinlabel {$0$} [ ] at 222 38
\endlabellist
\centering
\includegraphics[scale = 1]{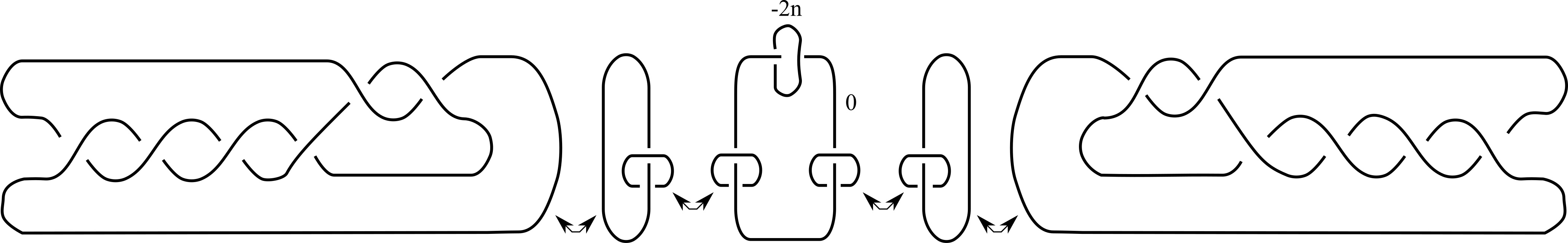}
    \caption{The homology sphere $Y_n$ as a splice}\label{Y2}.
\end{figure}
The Alexander-Conway polynomial of the stevedore knot is $\nabla(z) = 2z^2 - 1$. It follows that the Alexander-Conway polynomial of $K$ is $\nabla_K(z) = (2z^2 - 1)^2$. Clearly, $S^3_{1/0}(K) = S^3$, so the Casson invariant satisfies
$\lambda(S^3_{1/0}(K)) = 0$. Casson's surgery formula~\cite{akbulut-mccarthy} reads:
\[
\lambda(S^3_{1/(m+1)}(K)) - \lambda(S^3_{1/m}(K)) = \frac12 \nabla_K''(0),
\] 
so an induction argument gives $\lambda(Y_n) = \lambda(S^3_{1/2n}(K)) = -8n$. Because Casson's invariant changes sign under orientation reversal~\cite{akbulut-mccarthy}, the manifolds $Y_n$ with $n\neq 0$ do not admit any orientation-reversing diffeomorphisms. Furthermore, the stevedore knot is chiral so there is no diffeomorphism taking the complement of the mirror of the stevedore to the complement of the stevedore. If follows that up to isotopy, any diffeomorphism $\alpha\colon Y_n\to Y_n$ preserves the left torus in the JSJ decomposition as well as the complement of the mirror of the stevedore. 

It is known~\cite{knotinfo} that the mapping class group of the exterior of the stevedore knot is $\Z_2^2$ and up to isotopy any diffeomorphism acts as multiplication by $\pm 1$ in the $\R^2/\Z^2$ model of the boundary torus and as $(x,y,t)\mapsto (\pm x,\pm y,t)$ in a tubular neighborhood. Note that the generator of the $\Z_2^2$ that sends $(x,y,t)\to (x,-y,t)$ commutes with $f(n)$, while the generator that sends $(x,y,t)\to (-x,y,t)$ conjugates $f(n)$ to $f(n)^{-1}$. The other generators of $\diff(Y_n)$ are either twists along one of the tori or diffeomorphisms of the complement of the mirror of the stevedore. It's well-known that twists along a given torus yield an isomorphism $\pi_0(\diff(T^2 \times I,T^2 \times \partial I)) \cong \Z \oplus \Z$.  The main point is to isotope a given diffeomorphism of $T^2 \times I$ to make it level preserving; a detailed argument is written out in~\cite{myers:mcg}. Moreover, diffeomorphisms with disjoint support commute. 

It follows that for any $\alpha\in \pi_0(\diff(Y_n))$, either $\alpha$ commutes with $f(n)$ or conjugates it to its inverse. This distinction is encoded in the homomorphism $\delta$.
\end{proof}

We can now prove the first theorem.
\begin{proof}[Proof of Theorem~\ref{T:A}]
Recall that the goal is to find homology $3$-spheres $Y_n$ with topologically isotopic embeddings $j_p:Y_n \to S^4$ where the complements of $j_p(Y_n)$ are all diffeomorphic, but the embeddings are not smoothly isotopic. 
To this end, let $X_+(n)$ be the manifolds in Proposition~\ref{P:cork}, and set $Y_n=\partial X_+(n)$.  By the proposition, we may write $S^4$ as 
\[
X_+(n)\cup_{f(n)^p} -X_+(n) = X_+(n)\times\{\pm 1\}/\sim, \ \text{where}\  (y,1) \sim (f(n)^p(y),-1),
\]
oriented so that the orientation on $X_+(n)\times\{1\}$ agrees with the orientation on $X_+(n)$. 
Define $i_{p,n}\colon Y_n\to X_+(n)\cup_{f(n)^p} -X_+(n)$ by $i_{p,n}(y) = [y,1]$, and $j_{p,n} = \Phi_{p,n}\circ i_{p,n}$, where $\Phi_{p,n}$ is the diffeomorphism specified in Proposition~\ref{P:cork}. By construction, $S^4\setminus j_{p,n}(Y_n) \cong \text{int}(X_+(n)) \perp\!\!\perp \text{int}(X_-(n))$ and $X_-(n) \cong - X_+(n)$. The maps $j_{p,n}$ are defined for all $p$, but we will restrict to $p>0$ to get the collection required by the theorem.

Let
$\psi\colon S^4\to S^4$ be a diffeomorphism satisfying $\psi\circ j_{p,n} = j_{q,n}$. First assume that $\psi$ is orientation-preserving. Set $\alpha = \Phi_{q,n}^{-1}\circ\psi\circ\Phi_{p,n}$. As $\psi\circ j_{p,n} = j_{q,n}$  we see that $\alpha$ must take $X_+(n)\times \{1\}$ to $X_+(n)\times \{1\}$ or $X_+(n)\times \{-1\}$, and as $\alpha$ is orientation preserving, it must go to $X_+(n)\times \{1\}$. Let $\widehat{\alpha}$ be the diffeomorphism so that $\alpha([x,1]) = [\widehat{\alpha}(x),1]$. Similarly, let $\widecheck{\alpha}$ be the diffeomorphism so that $\alpha([x,-1]) = [\widecheck{\alpha}(x),-1]$.
For $y\in Y_n$ we have
\begin{align*}
[\widehat{\alpha}(y),1] &= \alpha([y,1]) \\
&= \alpha([f(n)^py,-1]) \\
&= [\widecheck{\alpha}f(n)^py,-1] \\
&= [f(n)^{-q}\widecheck{\alpha}f(n)^py,1] 
\end{align*}
and conclude that $\widehat{\alpha}(y) = f(n)^{-q}\widecheck{\alpha}f(n)^py$, so that $f(n)^q = \widecheck{\alpha}f(n)^p\widehat{\alpha}^{-1}$ on $Y_n$. 

Either one of $\delta(\widehat{\alpha})$ or $\delta(\widecheck{\alpha})$ is equal to $+1$, or both are $-1$.   In the former case there is an isotopy $H\colon I\times Y_n \to I\times Y_n$ so that $H_0(y) = f(n)^{q-p}(y)$ and $H_1(y) = \widecheck{\alpha}\widehat{\alpha}^{-1}(y)$. Now  identify $X_+(n)$ with $(I\times Y_n) \cup X_+(n)$ and define
\begin{small}
\[
\xi\colon  \left(E(2)\cs 2\cptwobar \setminus \text{\rm int}(X_+(n))\right)\cup_{f(n)^p} X_+(n) \to \left(E(2)\cs 2\cptwobar \setminus \text{\rm int}(X_+(n))\right)\cup_{f(n)^q} X_+(n),
\]
\end{small}
by
\[
\xi(w) = 
\begin{cases}
w, & w\in E(2)\cs 2\cptwobar \setminus \text{\rm int}(X_+(n))\\
H(t,y), & \text{\rm for} \ w = (t,y) \in I\times Y_n,\\
\widecheck{\alpha}\widehat{\alpha}^{-1}(w), & w\in X_+(n)
\end{cases}.
\]
One checks that this is a well-defined diffeomorphism to conclude that $p = q$ by part (1) of Proposition~\ref{P:cork} .
In the second case there is an isotopy $H\colon I\times Y_n \to I\times Y_n$ so that $H_0(y) = f(n)^{q+p}$ and $H_1(y) = \widecheck{\alpha}\widehat{\alpha}^{-1}(y)$. The argument from the first case then implies that $p = -q$.

If $\psi$ is orientation reversing, pre-compose with 
$\Phi_{p,n}\rho\Phi_{-p,n}^{-1}$ where
\[\rho\colon  X_+(n)\cup_{f(n)^{-p}} -X_+(n) \to X_+(n)\cup_{f(n)^{p}} -X_+(n)\]
is the map given by $\rho([x,\epsilon]) = [x,-\epsilon]$. This produces an orientation preserving diffeomorphism $\psi'\colon S^4\to S^4$ so that  $\psi'\circ j_{-p,n} = j_{q,n}\circ f(n)^{-p}$. 
Indeed,
\begin{align*}
    \psi'\circ j_{-p,n}(y) &= \psi\Phi_{p,n}\rho\Phi_{-p,n}^{-1}(y) \\
    &= \psi\Phi_{p,n}([y,-1]) \\
    &= \psi\Phi_{p,n}i_{p,n}f(n)^{-p}y \\
    &= \psi j_{p,n}f(n)^{-p}y = j_{q,n}\circ f(n)^{-p}(y).
\end{align*}
It follows that $\alpha' = \Phi_{q,n}^{-1}\circ\psi'\circ\Phi_{-p,n}$ is an orientation-preserving diffeomorphism
taking the $X_+(n)$-factor to the $X_+(n)$-factor. Defining $\widehat{\alpha}'$ and $\widecheck{\alpha}'$ as in the orientation-preserving case, we find $\widehat{\alpha}'(y) = f(n)^{-q}\widecheck{\alpha}f(n)^{-p}(y)$ and conclude $p = \pm q$. Thus, if we restrict our embeddings to the ones with $p>0$ there can be no intertwining orientation preserving or reversing diffeomorphism.
\end{proof}

\section{Stabilization and sums}\label{S:s}
In preparation for the proof of Theorem~\ref{T:B}, we now review some techniques for constructing interesting families of diffeomorphisms and embeddings.  The general principle underlying~\cite{auckly-ruberman:diffym} and the present paper is that stabilization (as defined below) can convert exotic smooth structures on $4$-manifolds into other kinds of exotic phenomena. This method can be iterated to build exotic families of diffeomorphisms, parameterized by spheres of increasing dimension. In turn, a fixed embedding of a submanifold can be composed with such a family of diffeomorphisms, yielding a family of embeddings that can often be shown to be exotic.

The simplest instance of this strategy (leading to exotic diffeomorphisms) is to start with $X$ and $X'$, distinct smooth structures on the same underlying topological $4$-manifold that are detected by a gauge theoretic invariant. Assume that $X\cs U\cong X'\cs U$ for some simple manifold $U$, typically $\sss$ or $\cptwo$.  (At present, this is known to hold with $U = \sss$ for many pairs of homeomorphic smooth closed $4$-manifolds, although not for all manifolds with boundary~\cite{kang:cork}.) After picking a specific identification of these two stabilized manifolds, an embedding or diffeomorphism that makes use of the stabilizing summand $U$ can often be shown to be exotic provided there is a way to recover information about smooth invariants of $X$ and $X'$ from the additional structure. 

Here is a first basic example, implicit in the early work of Donaldson~\cite{donaldson:polynomial}, using $U=\cptwo$. If $X\cs\cptwo\cong X'\cs\cptwo$, then after picking a specific identification of the stabilized manifolds, one may conclude that the sphere arising from the $\cpone\subset (\cptwo\setminus \text{pt})\subset X\cs\cptwo$ is topologically isotopic to, but smoothly distinct from the sphere arising from $\cpone\subset (\cptwo\setminus \text{pt})\subset X'\cs\cptwo$. (The topological isotopy is deduced from~\cite{perron:isotopy,quinn:isotopy}.) Indeed, if there was a diffeomorphism carrying one sphere to the other, blowing down the spheres would result in a diffeomorphism between $X$ and $X'$. A similar construction and argument may be made assuming $X\cs(S^2\times S^2)\cong X'\cs(S^2\times S^2)$, via surgery on embedded spheres of self-intersection $0$.  

An extension of this technique was introduced in~\cite{ruberman:isotopy} to study diffeomorphisms. The idea is that a self-diffeomorphism of $U$ gives rise to a  diffeomorphism of the stabilized manifold $X \cs U$ supported on $U$. For certain choices of $U$ and self-diffeomorphism, a gauge theoretic invariant of diffeomorphisms shows that if the diffeomorphism associated to the decomposition $X \cs U$ were isotopic to the one associated to the decomposition $X' \cs U$, then $X$ and $X'$ would have the same gauge theoretic invariants. 
As in the previous example, these smoothly non-isotopic diffeomorphisms are topologically isotopic.

\subsection{Submanifold sums}\label{S:ss}
{Many of the manifolds that go into our proof of Theorem~\ref{T:B} are constructed using 
\emph{submanifold sum} along a submanifold $\Sigma$, which we briefly review. Generally, we will write $N(\Sigma, X)$ for the normal bundle of $\Sigma$ in $X$, omitting the ambient manifold when it's clear. Assume that $\Sigma$ is embedded in $X$ and $V$ in such a way that the normal bundles are oppositely oriented, or in symbols $N(\Sigma,V) = N(\Sigma,X)^\vee$ where $N^\vee$ denotes the bundle with orientation opposite to $N$.  The submanifold sum is given by
\begin{equation}\label{E:sum}
X\cs_\Sigma V = (X\setminus N(\Sigma))\cup(V\setminus N(\Sigma))/\sim,
\end{equation}
where $N(\Sigma,X) \setminus\Sigma$ is identified to $N(\Sigma,V)\setminus\Sigma$ via a map that acts as $v\mapsto |v|^{-2}v$ on the fibers. }

The submanifold sum is a fairly general operation and depends only on the embedding of the normal bundle up to isotopy. In dimension $4$, this operation is most often done along a surface, typically denoted by $\Sigma$, but we will also consider sums along $3$-manifolds (typically denoted by $Y$). For sums along $3$-manifolds, (as pictured in Figure~\ref{F:cd1sum} below), the isotopy type is determined by orientations and the isotopy type of the embedding of the submanifold. In section~\ref{S:inv}, we will consider sums along families and so in the following remark discuss the potential dependence on choices.  The bottom line is that in the family setting, the case of codimension one submanifold sums is a little simpler than codimension two.

In codimension one, since $Y$ and $V$ are both oriented, the normal bundle is $N(Y) = Y\times [-1,1]$ and this identification is unique up to a choice in the orientation-preserving automorphism group of the normal bundle, which is contractible.
Therefore $\emb(Y,V)$ is homotopy equivalent to $\emb(N(Y),V)$. On the other hand, if $\Sigma$ has codimension two, the embeddings of the normal bundle extending a given embedding of the submanifold differ up to isotopy by a bundle automorphism. The space of bundle automorphisms may be identified with maps from $\Sigma$ into the circle. Its path components, each of which are contractible, are given by $H^1(\Sigma; \Z)$.

There is a minor variation of the submanifold sum in codimension one that we will use, and we explain the notation here. If we write $X \cs_Y (-V)$, then this should be taken as an oriented manifold where the orientation agrees with the given orientation on $X$.  More formally, when $Y$ separates both $X$ and $V$, this can be written as follows. Label the components of $X \setminus Y$ as $X_\pm$ so that $Y \times [-1,0) \subset X_{-}$, and do the same for $V$. Then $X\cs_Y (-V)$ would be written as
\[  
\left( (X_{-} \cup V_{+})/\sim \right) \coprod \left( (X_{+} \cup V_{-})/\sim \right)
\]
where `$\sim$' indicates the same inversion on $N(Y)$ as in the usual submanifold sum. This process is depicted in Figure~\ref{F:cd1sum}; we also do this in families.  

Using these ideas, here is how to build many collections of exotic embeddings.  The construction, and the related construction of higher-parameter families of embeddings, uses as building blocks some $4$-manifolds with specific properties. We describe a particular collection of such manifolds below in Section~\ref{S:blocks}, but it's easier to see the overall principle if one doesn't worry too much about the specific ingredients.

Let $X_p$ be a collection of exotic $4$-manifolds that are distinguished by some gauge theoretic invariant $I$ (say $I = D$ for Donaldson or $I = \sw$ for Seiberg-Witten) so that
\[
I(X_p) = I(X_q) \ \text{implies} \ p = q.
\]
Suppose that each $X_p$ has  an embedding
 $\nu_p\colon N(\Sigma)\hookrightarrow X_p$, here $\Sigma$ is a surface. By convention, exotic manifolds are homeomorphic; in this case we further assume that there are homeomorphisms that respect the embedding of $N(\Sigma)$, i.e., 
\[
\psi_p\colon X_p\to X_0, \ \text{so that} \ \psi_p\circ\nu_p = \nu_0.
\]
We now let $U$ be a stabilizing manifold so that $X_p\cs U$ is diffeomorphic to $X_0\cs U$, where the diffeomorphism respects the embeddings $\nu_p$. It is not hard to find such diffeomorphisms; see for example~\cite{auckly-kim-melvin-ruberman:isotopy}. 
The point of the embedded copy 
of $\Sigma$ is to allow a submanifold sum of a `buffer' manifold $V$ containing a copy of $\Sigma$, giving the construction greater flexibility.

Suppose further that the collection $\{X_p\cs_\Sigma V\}$  is exotic, and continues to be distinguished by the invariant $I$. This would be the case for $V = \cptwobar$ when the submanifolds sum is just connected sum (along a $3$-sphere) and $I$ is a Seiberg-Witten or Donaldson invariant.
Let 
\[
\iota\colon Y\hookrightarrow V\setminus\Sigma
\]
be a smooth, separating embedding of a $3$-manifold so that
$V\setminus\iota(Y) = V_0\perp\!\!\perp V_1$ with $\iota(\Sigma)\subset V_0$. Set $\widehat{V} = V \cs U$ where the connected sum takes place in $V_1$, so that $\widehat{V} =\overline{V_0}\cup_Y(\overline{V_1\cs U})$.
The assumption that the $X_p$ become diffeomorphic after connected sum with $U$, respecting the embedding of $\Sigma$, implies that there is a diffeomorphism
\[ \varphi_p\colon X_p\cs_\Sigma\widehat{V} \to X_0\cs_\Sigma\widehat{V}.\]
We also take $\varphi_0 = 1$ and $\psi_0 = 1$.
It will be convenient to have  $\varphi_p$ be homotopic to $\psi_p\cs_\Sigma 1_{\widehat{V}}$ for all $p$. Since we are dealing with simply connected $4$-manifolds, two maps are homotopic if and only if they induce the same isomorphism on homology~\cite{cochran-habegger:homotopy}. By a classic theorem of Wall~\cite{wall:diffeomorphisms} one can compose an initial choice of $\varphi_p$ with a diffeomorphism to ensure that this is the case if $X_p$ is indefinite and $U=\sss$. In this case a result of Kreck~\cite{kreck:isotopy} states that $\varphi_p(\psi_p\cs_\Sigma 1_{\widehat{V}})^{-1}$ is pseudoisotopic to the identity, and so work of Perron~\cite{perron:isotopy} and Quinn~\cite{quinn:isotopy} implies that it is topologically isotopic to the identity.
 Notice that both $\varphi_p$ and $(\psi_p\cs_\Sigma 1_{\widehat{V}})^{-1}$ restrict to smooth maps on $\widehat{V}\setminus\Sigma$ and in particular are smooth on $Y$. Thus we can define {\em smooth} embeddings
\[j_p\colon Y\hookrightarrow X_0\cs_\Sigma\widehat{V},\]
by $j_p = \varphi_p(\psi_p\cs_\Sigma 1_{\widehat{V}})^{-1}\circ\iota$, and immediately conclude that the maps $j_p$ are topologically isotopic 
o one another.
\begin{figure}[ht]
\labellist
\pinlabel {$X_p$} [ ] at 60 98
\pinlabel {$Y$} [ ] at 180 72
\pinlabel {$\Sigma$} [ ] at 104 89
\pinlabel {$\Sigma$} [ ] at 173 89
\pinlabel {$V_0$} [ ] at 148 100
\pinlabel {$V_1$} [ ] at 195 100
\pinlabel {$U$} [ ] at 250 98
\pinlabel {$Y$} [ ] at 80 -10
\pinlabel {$j_p(Y)$} [ ] at 256 -10
\pinlabel {$X_p\cs_\Sigma \widehat{V}$} [ ] at 32 23
\pinlabel {$\widehat{V}$} [ ] at 105 25
\pinlabel {$X_0\cs_\Sigma \widehat{V}$} [ ] at 207 23
\pinlabel {$\varphi_p$} [ ] at 160 33

\endlabellist
\centering
\includegraphics[scale = 1]{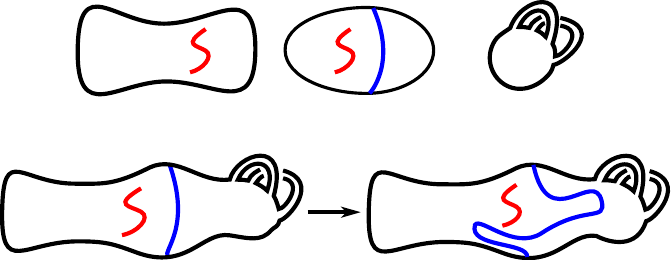}
\vspace*{2ex}
  \caption{Construction of the embedding $j_p$}\label{F:subsum}
\end{figure}
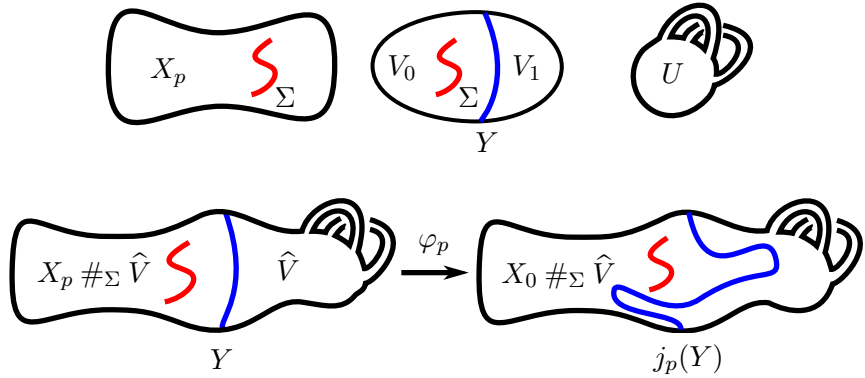

It remains to show that the embeddings $j_p$ are smoothly distinct.
Assuming that there is a diffeomorphism $f\colon X_0\cs_\Sigma\widehat{V} \to X_0\cs_\Sigma\widehat{V}$ so that $f\circ j_p = j_q$, one would conclude that 
\[\overline{\left((X_0\cs_\Sigma\widehat{V})\setminus j_p(Y)\right)_0}\bigcup_Y\overline{V_1}\cong \overline{\left((X_0\cs_\Sigma\widehat{V})\setminus j_q(Y)\right)_0}\bigcup_Y\overline{V_1}.\]
Here the $0$-subscript refers to the component disjoint from $U\setminus\text{pt}$. Using extensions of suitable restrictions of $\varphi_p$ and $\varphi_q$ one would further conclude that 
$X_p\cs_\Sigma{V}\cong X_q\cs_\Sigma{V}$, implying that $p = q$.

\begin{remark}
Theorem~\ref{T:A} draws a finer distinction between smoothly knotted embeddings. The complements of the embeddings in that theorem are independent of the choice of embedding. The fact that the embeddings are not smoothly equivalent follows from  the existence of non-extendible diffeomorphisms on a manifold with boundary. The paper~\cite{konno-mukherjee-taniguchi:codim1} constructs a number of other interesting examples of smoothly knotted embeddings based on a different exotic phenomenon, namely the non-smoothablility of certain homeomorphisms.
\end{remark}

\subsection{Higher parameter families}\label{s:recursion}The construction of exotic diffeomorphisms via stably equivalent smooth structures combines with further stabilizations by a manifold $U$ (such as $\cptwo$ or $\sss$, often combined with a buffer manifold as in the previous subsection) to yield exotic spherical families of diffeomorphisms. The construction is recursive, raising the dimension of the parameterizing sphere by one at each step. At the base level, let $X$ and $X'$ be distinct smooth manifolds with
a  diffeomorphism $\varphi\colon X\cs U\to X'\cs U$ that induces the same isomorphism on homology as $\psi \cs 1_U$ for a homeomorphism $\psi\colon X \to X'$ and preserves the embeddings of $\Sigma$, and write $Z_0$ for $X\cs U$. Let $f\colon U\to U$ be a diffeomorphism acting non-trivially on $H_2(U)$, defined using $2$-spheres embedded in $U$. A prototype might be $U = \cptwo$, with $f$ defined as complex conjugation, or equivalently as a reflection in the $2$-sphere $\cpone$. In examples, the construction is slightly more elaborate.

In this setting, it is reasonable to conjecture that 
\begin{equation}\label{callmeal}
\alpha^0 = [\varphi,1\cs f] = \varphi\circ(1_{X'}\cs f)\varphi^{-1}(1_X\cs f)^{-1}\colon Z_0\to Z_0,
\end{equation}
is an exotic diffeomorphism. Indeed the original construction of~\cite{ruberman:isotopy} was of this form, and there are now a number of additional examples that use the same technique.  It is worth noting that the commutator  in Equation~\eqref{callmeal} has embedded in it a crucial switch from $X$ to $X'$; other commutator-like constructions of this sort will be used without further comment.  A particular choice of $\alpha^0$ and $U$ is given in Section~\ref{S:blocks}.
To describe the recursion, imagine that 
 $\alpha^k\colon S^k\to \diff(Z_k)$ represents a non-trivial element of $\pi_k(\diff(Z_k))$, and
  that $\alpha^k\cs 1_U$ represents the trivial element of $\pi_k(\diff(Z_k\cs U))$. In what follows, we will be assuming that all of the families (and homotopy of families) of diffeomorphisms are relative to $N(\Sigma)$, but we omit it from the notation for sanity's sake.

The presence of the summand of $U$ in $\widehat{V}$ implies that 
$\alpha^k\cs_\Sigma 1_{\widehat{V}}$ represents the trivial element of $\pi_k(\diff(Z_k\cs_\Sigma\widehat{V}))$. Let
$F\colon I\times S^k\to \diff(Z_k\cs_\Sigma\widehat{V})$ be the corresponding null-homotopy. Set $Z_{k+1} = Z_k\cs_\Sigma\widehat{V}$ and let $\beta_0\colon \widehat{V} = V\cs U \to \widehat{V}$ be the diffeomorphism constructed as in the base case, {in other words, $\beta_0=1_V \cs_\Sigma f$}.
We then define $\alpha^{k+1}\colon I\times S^k\to\diff(Z_{k+1})$ by $\alpha^{k+1} = [F,1_{Z_k}\cs_\Sigma\beta_0]$. In other words given $(t,v)\in I\times S^k$ and $z\in Z_{k+1}$ we have
\begin{equation}\label{callmealk}
\alpha^{k+1}(t,v)(z) = F(t,v)(1_{Z_k}\cs_\Sigma\beta_0)F(t,v)^{-1}(1_{Z_k}\cs_\Sigma\beta_0)^{-1}(z).
\end{equation}
Since $F(0,v) = 1_{Z_k}$, and $F(1,v) = \alpha^k\cs_\Sigma 1_{\widehat{V}}$ commutes with $1_{Z_k}\cs_\Sigma\beta_0$, we see that $\alpha^{k+1}|_{\partial I\times S^k} = 1_{Z_{k+1}}$ so it induces a map on the
suspension of $S^k$, i.e., on $S^{k+1}$. In fact it also restricts to the identity on $I\times \{v_0\}$ so it also induces a map on the reduced suspension. 

This process, which gives rise to families of diffeomorphisms $\alpha^k$ as well as associated families of embeddings, is represented schematically in Figure~\ref{F:cartoon}.
In this figure, we have represented a diffeomorphism or family of diffeomorphisms by a red curve and a blue curve. One can imagine that the diffeomorphism takes the red to the blue (or that a spherical family moves these curves around in some fashion.) The smooth isotopy $F$ does not exist without the $U$ summand in $\widehat{V}$, so one imagines the isotopy would give a family moving the blue curve to the red. However, this must pass through the right-most copy of $U$, as represented by a pair of handles on the far right of the picture, and this in some sense is what makes it an `interesting' isotopy.

In section~\ref{S:inv}, we will discuss invariants that can be used to show that families constructed in this way from suitable {initial collections of $4$-manifolds} are non-trivial. 

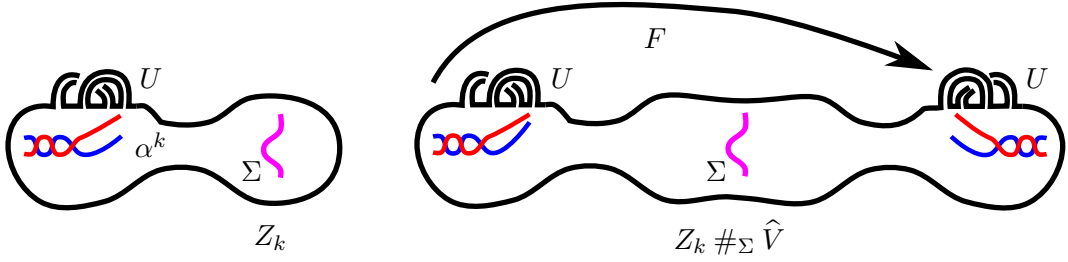
\begin{figure}[h]
\labellist
\pinlabel {$Z_k$} [ ] at 100 -10
\pinlabel {$\alpha^k$} [ ] at 55 25
\pinlabel {$\Sigma$} [ ] at 93 15
\pinlabel {$\Sigma$} [ ] at 268 15
\pinlabel {$Z_k\cs_\Sigma \widehat{V}$} [ ] at 273 -10
\pinlabel {$F$} [ ] at 245 65
\pinlabel {$U$} [ ] at 55 50
\pinlabel {$U$} [ ] at 210 50
\pinlabel {$U$} [ ] at 389 50
\endlabellist
\centering
\includegraphics[scale = 1]{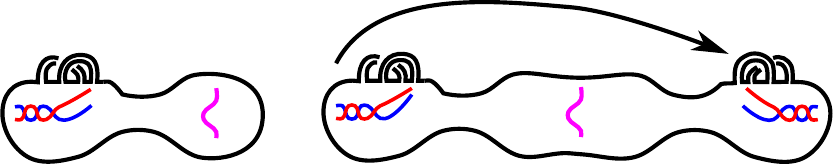}
\vspace*{1ex}
  \caption{Construction cartoon}\label{F:cartoon}
\end{figure}

\section{Families and invariants}\label{S:inv}
Families of diffeomorphisms, manifolds, and embeddings are closely related. For instance, a family of manifolds diffeomorphic to $Z$ is, by definition, a bundle over a base $B$ with fiber diffeomorphic to $Z$ and structure group $\diff(Z)$, and hence is equivalent to a map from $B$ to the classifying space $B\diff(Z)$.  We discuss some of the relations between these three types of families here and then turn to the definition of various invariants that are useful to distinguish these families. The connection between families of diffeomorphisms and bundles is most visible in the case of spherical families.

\begin{definition}
     If $\alpha\colon S^k\to \diff(Z)$ is a spherical family of diffeomorphisms, the family of spaces associated to $\alpha$ via the \emph{clutching construction} is
    \[ \mathcal{Z}[\alpha] = \{\pm 1\}\times D^{k+1} \times Z/\sim,\]
    where $(1,v,z)\sim (-1,v,\alpha(v)(z))$ for $|v| = 1$. If we view $S^{k+1} = \{\pm 1\}\times D^{k+1}/\sim$, then there is a natural fiber bundle projection $\mathcal{Z}[\alpha]\to S^{k+1}$.   
\end{definition}

We have the following facts about the clutching construction.
\begin{lemma}\label{fam} 
{\bf (a)} Let $\alpha\colon S^k\to \diff(Z)$ be a spherical family of diffeomorphisms and $F\colon I\times S^k\to\diff(Z)$ a contraction to the constant family $\bone_Z$ given by $1_Z$ at each point of $S^k$.  Then there is a family isomorphism
$
    \Phi\colon \mathcal{Z}[\bone_Z]\to \mathcal{Z}[\alpha]
$
given by
    \[\Phi([\epsilon,v,z]) = \begin{cases} [1,v,z], & \text{if} \ \epsilon = 1, \\
    [-1,v,F(|v|,v)(z)], & \text{if} \ \epsilon = -1.
    \end{cases}
    \]
{\bf (b)} Let $\beta\colon Z'\to Z$ be a diffeomorphism.   Then there is a family isomorphism
\[
\xi\colon \mathcal{Z'}[\beta^{-1}\alpha\beta]\to \mathcal{Z}[\alpha],
\]
given by $\xi([\epsilon,v,w]) = [\epsilon,v,\beta(w)]$.
\end{lemma}

\begin{remark}
    The lemma follows by a direct check of the definitions. The second isomorphism $\xi$ would also work if $\beta$ was a spherical family of diffeomorphisms.
\end{remark}

We now describe how to use these ideas to construct and detect interesting families of embeddings.
The construction of families begins the same way as the construction of smoothly knotted codimension one embeddings in Section~\ref{S:ss}.
The ingredients that go into the construction are 
\begin{enumerate}
    \item A collection of homeomorphic smooth manifolds $X_p$ together with a family of invariants $I_{\cc_q}$, parameterized by data on the manifold, so that $I_{\cc_q}(X_p) = \gamma_p\delta_{pq}$ for some non-zero constants $\gamma_p$.  In the Donaldson case,  $\cc_q$ describes a certain $\SO(3)$ bundle, while in Seiberg-Witten case, it specifies a $\spinc$ structure.
    \item Embeddings $\nu_p\colon N(\Sigma_1\perp\!\!\perp\Sigma_2)\hookrightarrow X_p$. 
    \item A smooth buffer manifold $V = \overline{V_0}\cup_Y\overline{V_1}$ with $N(\Sigma_2)\hookrightarrow V_0$.
    \item A collection of manifolds $Z_k$ together with topologically trivial families of diffeomorphisms  $\alpha^k\colon S^k\to\diff(Z_k,N(\Sigma_1))$  so that $\alpha^k\cs 1_U$ smoothly contracts and the corresponding $I^{H_k\diff}_{\cc_q}(\alpha^k\cs_{\Sigma_1}1_{X_p\cs_{\Sigma_2}V}) = \gamma'_p\delta_{pq}$ for some non-zero constants $\gamma'_p$ and data $q$. 
\end{enumerate}
As before, we will have a manifold $U$ (typically $S^2 \times S^2$ or $(S^2 \times S^2)\cs \cptwobar$) and write $\widehat{V} = V \cs U$ where the connected sum takes place in $V_1$.

Notice that we just need one spherical family of diffeomorphisms for each dimension and use a large collection of exotic embeddings of a submanifold to build a large collection of families of embeddings in each dimension.

We now provide a recipe for combining these ingredients into examples of the proper form. We will list some assumptions along the way. Let $\psi_p\colon X_p\to X_0$ be homeomorphisms so that $\psi_p\circ\nu_p = \nu_0$, and $\varphi_p\colon X_p\cs_{\Sigma_2}\widehat{V}\to X_0\cs_{\Sigma_2}\widehat{V}$ be diffeomorphisms homotopic to $\psi_p\cs_{\Sigma_2}1_{\widehat{V}}$.
This implies that there is a topological isotopy $G^p\colon I\to\homeo(X_0\cs_{\Sigma_2}\widehat{V})$ with 
$G_1^p = \varphi_p(\psi_p\cs_{\Sigma_2}1_{\widehat{V}})^{-1}$ and $G_0^p = 1_{X_0\cs_{\Sigma_2}\widehat{V}}$. In the examples  used in this paper, the isotopy $G^p$ comes from the `key stable isotopy' in~\cite[\S 2]{auckly-kim-melvin-ruberman:isotopy} and so may be chosen to be the identity on $N(\Sigma_1)$. This additional property holds in many other settings; for example if the complement of $\Sigma_1$ has trivial~\cite{sunukjian:surfaces}, infinite cyclic~\cite{conway-powell:z-surfaces} or in some cases finite cyclic~\cite{kim-ruberman:surfaces} fundamental group.
Define $j_p\colon Y\to X_0\cs_{\Sigma_2}\widehat{V}$ by $j_p = \varphi_p(\psi_p\cs_{\Sigma_2}1_{\widehat{V}})^{-1}\iota$. Notice that this is smooth because the restriction of $\varphi_p(\psi_p\cs_{\Sigma_2}1_{\widehat{V}})^{-1}$ to the image of the embedding $\iota\colon Y\to \widehat{V}\setminus\Sigma_2$ is smooth.
Set
\[
Z_{k}' = Z_k\cs_{\Sigma_1} \left(X_0\cs_{\Sigma_2}\widehat{V}\right).
\]
Notice that $X_0\cs_{\Sigma_2}\widehat{V}$ contains a $U$-summand. This implies that \hfill\newline $\alpha^k\cs_{\Sigma_1}1_{X_0\cs_{\Sigma_2}\widehat{V}}\colon S^k\to\diff(Z_{k}\cs_{\Sigma_1} X_0\cs_{\Sigma_2}\widehat{V})$ smoothly contracts. By Lemma~\ref{fam} we conclude that there is a bundle isomorphism
\[
\Phi^k\colon \mathcal{Z}[\alpha^k\cs_{\Sigma_1}1_{X_0\cs_{\Sigma_2}\widehat{V}}]\to S^{k+1}\times Z_k'.
\]
Since $\alpha^k\cs_{\Sigma_1}1_{X_0\cs_{\Sigma_2}\widehat{V}}$ acts trivially on $\widehat{V}\setminus\Sigma_2$, the embedding $j_p$ extends to a family of embeddings $\underline{j_p}\colon S^{k+1}\times Y \to \mathcal{Z}[\alpha^k\cs_{\Sigma_1}1_{X_0\cs_{\Sigma_2}\widehat{V}}]$ given by $\underline{j_p}([\epsilon,v],y) = [\epsilon,v,j_p(y)]$.
Now define 
\[
J^{k+1}_p\colon  S^{k+1}\to\emb(Y,Z_{k}'), \ \text{by} \ J^{k+1}_p([\epsilon,v])(y) = \text{pr}_{Z_{k}'}\Phi^k(\underline{j_p}([\epsilon,v],y))
\]

\begin{lemma}\label{L:same}
The classes $[J^{k+1}_p]$ and $[J^{k+1}_0]$$\pi_{k+1}(\emb^{\Top}(Y,Z_{k}'))$.
\end{lemma}

\begin{proof}
As discussed above, we are working in a setting where the topological isotopy $G^p$ has support in $X_0 \cs_{\Sigma_2}\widehat{V}\setminus\Sigma_1$ so extends to define a homotopy 
\[
\underline{G}^p\colon I\to \emb(Y,Z_{k}'), \ \text{given by} \ \underline{G}_t^p(y) = \text{pr}_{Z_{k}'}\Phi_k([\epsilon,v,G^p_tj_0(y)]).
\qedhere
\] 
\end{proof}

There is a family version of the submanifold sum, both for sums along surfaces and along $3$-manifolds. As we will see, it interacts well with our constructions of families of diffeomorphisms.
\begin{definition}
Let $J\colon \Xi\to \emb(N(Y),Z)$ and $K\colon \Xi\to \emb(N(Y),W)$ be given and set
\[\mathcal{Z}[\![J,K]\!] = \left((\Xi\times Z)\setminus J(Y)\right)\perp\!\!\perp \left((\Xi\times W)\setminus K(Y)\right)/\sim\] where the fibers of the normal bundles are identified via inversion in the unit sphere as in~\eqref{E:sum}.
\end{definition}

\begin{remark}
    Note that if $\Xi$ is connected, then we are getting  an embedding of $Y$ in a manifold diffeomorphic to $Z\cs_Y W$. But when done in the family setting, the resulting family may well be non-trivial. One could go further and define a family notion of submanifold sum for family embeddings into non-trivial families. When the $K$-family of embeddings is clear from the context, we will remove it from the notation.  
\end{remark}

Let $Z$ contain a summand (along a surface) of $\widehat{V}$ as in our description of smoothly knotted embeddings in Section~\ref{S:ss}, and let $\iota\colon N(Y)\hookrightarrow V$ be the model embedding and 
$J\colon \Xi^{k+1}\to\emb(N(Y),Z)$ be a family of embeddings representing a class in $H_{k+1}(\emb(N(Y),Z))$. The family submanifold sum
\[
{(\Xi^{k+1}\times Z)\cup_J (\Xi^{k+1}\times -V)}
\]
will consist of two components. 
Let $\mathcal{Z}[\![J]\!]$ denote the component that does not include the copy of $U$ in each fiber. We call it the preferred component. The component with $U$ will be called the discarded component. See Figure~\ref{F:cd1sum}, which omits the parameter space $\Xi$.

\vspace*{1ex}
\begin{figure}[ht]
\labellist
\pinlabel {Discarded component} [ ] at 325 -10
\pinlabel {Preferred component} [ ] at 80 -10
\endlabellist
\centering
\includegraphics[scale = 1]{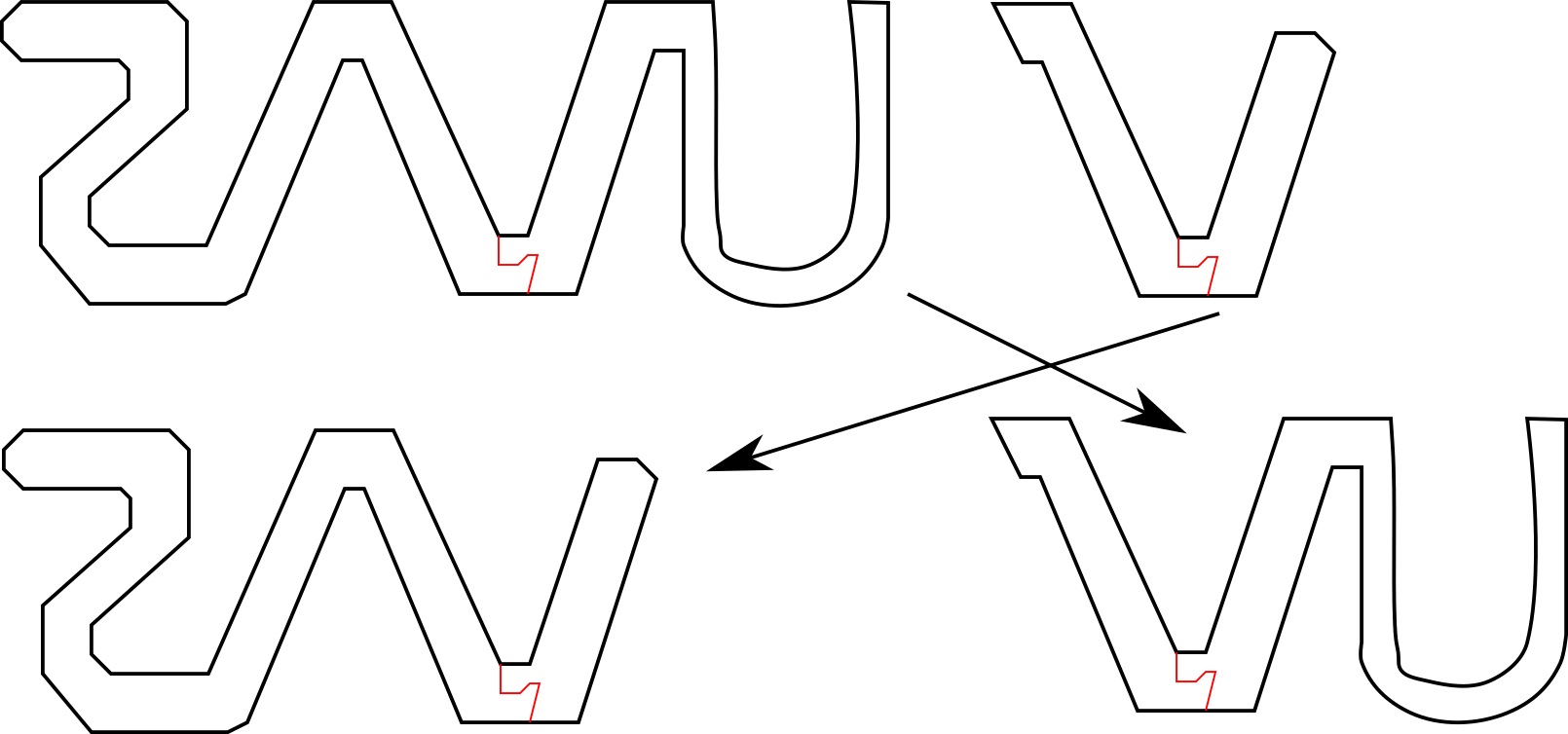}
\vspace*{2ex}
  \caption{Codimension one submanifold sum}\label{F:cd1sum}
\end{figure}

Let us assume that that the basic $I_{\cc_q}$
extends to an invariant of families of manifolds, denoted $\mathbb{I}_{\cc_q}$. Given this assumption, the invariant for a family of embeddings is given by
\[
I^{H_{k+1}\emb}_{\cc_q}(J) = \mathbb{I}_{\cc_q}(\mathcal{Z}[\![J]\!]).
\]

\begin{lemma}
The classes $[J^{k+1}_p]$ are independent in  $\pi_{k+1}(\emb(Y,Z_{k}')) $.
\end{lemma}

\begin{proof}
Notice that $\Phi^k\colon \mathcal{Z}[\alpha^k\cs_{\Sigma_1}1_{X_0\cs_{\Sigma_2}\widehat{V}}]\to S^{k+1}\times Z_{k}')$
satisfies 
\[\Phi^k\underline{j_p}([\epsilon,v],y) = [\epsilon,v,J^{k+1}_p([\epsilon,v])(y)].
\]
The diffeomorphism $1_{Z_k}\cs_{\Sigma_1}\varphi_p\colon Z_k\cs_{\Sigma_1}X_p\cs_{\Sigma_2}\widehat{V} \to Z_k\cs_{\Sigma_1}X_0\cs_{\Sigma_2}\widehat{V}$ combined with  Lemma~\ref{fam} then gives a family isomorphism
\[
\xi\colon \mathcal{Z}[\alpha^k\cs_{\Sigma_1}1_{X_0\cs_{\Sigma_2}\widehat{V}}] \to \mathcal{Z}[\alpha^k\cs_{\Sigma_1}1_{X_p\cs_{\Sigma_2}\widehat{V}}],
\]
so that $\alpha^k\cs_{\Sigma_1}1_{X_0\cs_{\Sigma_2}\widehat{V}}j_p = \alpha^k\cs_{\Sigma_1}1_{X_p\cs_{\Sigma_2}\widehat{V}}\iota$.
  It follows that
\begin{align*}
\mathcal{Z}[\![J^{k+1}_p]\!] &= \text{preferred component}\left((S^{k+1}\times Z_{k}'))\cs_{J^{k+1}_p}(S^{k+1}\times -V)\right)\\
&\cong \text{preferred component}\left(\mathcal{Z}[\alpha^k\cs_{\Sigma_1}1_{X_0\cs_{\Sigma_2}\widehat{V}}]\cs_{\underline{j_p}}(S^{k+1}\times -V)\right) \\
&\cong \text{preferred component}\left(\mathcal{Z}[\alpha^k\cs_{\Sigma_1}1_{X_p\cs_{\Sigma_2}\widehat{V}}]\cs_{\underline{\iota}}(S^{k+1}\times -V)\right) \\
&\cong \mathcal{Z}[\alpha^k\cs_{\Sigma_1}1_{X_p\cs_{\Sigma_2}{V}}]
\end{align*}
Thus,
\begin{align*}
I^{H_{k+1}\emb}_{\cc_q}(J^{k+1}_p) &= \mathbb{I}_{\cc_q}(\mathcal{Z}[\![J^{k+1}_p]\!]) \\
&= \mathbb{I}_{\cc_q}(\mathcal{Z}[\alpha^k\cs_{\Sigma_1}1_{X_p\cs_{\Sigma_2}{V}}]) \\
&= I^{H_{k}\diff}_{\cc_q}(\alpha^k\cs_{\Sigma_1}1_{X_p\cs_{\Sigma_2}{V}}) = \gamma_p'\delta_{pq}.
\end{align*}
This is sufficient to show that the families of embeddings are independent in the homotopy group.

We can slightly recast the invariant $I^{H_{k+1}\emb}_{}$. By making $\cc_q$ an argument we obtain
\[
I^{H_{k+1}\emb}_{(\cdot)}\colon H_{k+1}(\emb(Y,Z_{k}'))) \to \text{Maps}(\text{Set of }\cc_q\text{ values},\Z).
\]
Thus, $I^{H_{k+1}\emb}_{(\cdot)}\left(\ker\left[\pi_{k+1}(\emb(Y,Z_{k}')))\to \pi_{k+1}(\emb^{\Top}(Y,Z_{k}')))\right] \right)$
is a free group of rank equal to the number of initial exotic manifolds $X_p$. As it is free there is a right inverse
$s$ to $I^{H_{k+1}\emb}_{(\cdot)}$ from this group to the kernel 
giving the required summand. The same argument generates the summand on the kernel at the level of homology.
\end{proof}

\section{Specific choices}\label{S:blocks}
To make the outline in the prior section into a proof, one must make specific choices for the manifolds, diffeomorphisms, and invariants. We will see that there are many different choices for manifolds $X_p$, $U$, and $V$, and the diffeomorphism $f:U \to U$ that can be used as the basis of our constructions. In this section, we give some specific choices and specify the manifolds $Z$ appearing in the statement of Theorem~\ref{T:B}.  All of them are built in some fashion out of the elliptic surface $E(2)$, which has $q_{E(2),\cc,\beta} = -1$ with the homology orientation and class $\cc$ described in \cite{donaldson-kronheimer}. It also has $\text{SW}(E(2),0) = 1$. 

The surface $E(2)$ has three embedded copies~\cite{gompf-mrowka} of the Gompf nucleus~\cite{gompf:nuclei}, labeled $N^{(1)}$, $N^{(2)}$ and $N^{(3)}$, whose fibers are denoted $T^{(1)}$, $T^{(2)}$ and $T^{(3)}$. (We retain the convention from~\cite{auckly-ruberman:diffym} of denoting various specific manifolds with bold letters; these are not quite the same as the manifolds appearing in this paper, which are in ordinary font.) Set $\xx = E(2)\cs\cptwobar$. 
We also define $\yy$ by a cork twist  on $\xx$ using the cork discussed below in Example~\ref{E:selman}; in particular, it decomposes as a connected sum of copies of $\cptwo$ and $\cptwobar$ and hence has trivial gauge-theoretic invariants. This is the original example of an effective cork discovered by Akbulut. It is important to notice that the cork is disjoint from the $N^{(1)}$ nucleus; see~\cite[Figure 7]{auckly-kim-melvin-ruberman:isotopy}.

By Freedman's work there is a homeomorphism $\psi:\xx\to\yy$ relative to $N^{(1)}$. Furthermore, if we set $\mm = \cptwo \cs 2 \cptwobar \cong S^2 \times S^2 \cs \cptwobar$ there is a diffeomorphism $\varphi:\xx\cs\mm \to \yy\cs\mm$ relative to $N^{(1)}$ inducing the same map on homology. We use $X'=\xx$, $X=\yy$, and $Z_0=\yy\cs\mm$ in the construction outlined in section~\ref{s:recursion}. 

Corresponding to the three summands in the stabilizing manifold $U = \mm$, there are $2$-spheres $H$, $E_1$, and $E_2$ of self-intersection $1$, $-1$, and $-1$, respectively. Any sphere $S$ of self-intersection $\pm 1$ in a $4$-manifold determines a reflection $\rho^S$, and we define
\[
f = \rho^{E_1}\rho^{H+E_1 + E_2}.
\]
As described in equations~\eqref{callmeal} and~\eqref{callmealk} in Section~\ref{S:ss}, such a diffeomorphism gives rise to a diffeomorphism $\alpha^0:Z_0\to Z_0$, and by a recursive process, to spherical families $\alpha^k$ of diffeomorphisms on the related manifolds $Z_k$. This is exactly the rank one case of the construction of exotic families of diffeomorphisms presented in \cite{auckly-ruberman:diffym}. In addition, the commutator construction gives topological contractions and smooth stabilized contractions. The formula for each of these is
\begin{equation}\label{E:r=1}
\begin{aligned}
\alpha^{k+1}& = [F^{k},1_{\zz_{k}}\cs_T\alpha]\\
G^{k+1}& = [F^{k},1_{\zz_{k}}\cs_TG]\\
F^{k+1}& = [F^{k},1_{\zz_{k}}\cs_TF].
\end{aligned}
\end{equation}
We will only need the $\alpha^k$ and $F^k$ in this paper.

The basic idea to create families of embeddings is to fiber sum a family of manifolds constructed via the clutching construction with the $\alpha^k$ families of diffeomorphisms with the collections of exotic embeddings described in section~\ref{S:ss}. Schematically we are just gluing a copy of Figure~\ref{F:subsum} onto the end of a copy of Figure~\ref{F:cartoon}.

The buffer manifold $V$ is defined to be the blowup $E(2)\cs \cptwobar$ (this is the same as $\xx$ but we are using it differently). We take $\Sigma_2 = T^{(2)}$ the fiber in the second nucleus.
Any separating $3$-manifold $Y$ that embeds in $V$ disjoint from a fiber will satisfy the conclusions of the theorem and have large families of embeddings in the manifolds $Z$ constructed below. 
Some popular choices for $Y$ might be $S^3$ (as the boundary of a ball), the Brieskorn sphere $\Sigma(2,3,11)$ (boundary of a nucleus), the Brieskorn sphere $\Sigma(2,3,7)$ (boundary of a Milnor fiber~\cite{fs:237}), or the boundary of the Akbulut-Mazur cork~\cite{akbulut:contractible}. Note that for this last choice there are two distinct embeddings of $Y$, differing by reparameterization using the cork twist. In any of these cases let $V_0$ be the component of $V\setminus Y$ that contains $\Sigma_2 = T^{(2)}$ and let $V_1$ be the other component.

To construct the manifolds $X_p$ we will use log transforms and fiber sums. Let $E(2)_{p;q;r}$ be the result of performing log transforms of multiplicity $p, q$, and $r$ on the fibers in the $N^{(i)}$ nuclei for $i=1,2,3$, where by convention `$\infty$' means that no surgery has been done. $E(2)_{p;q;r}$  inherits a homology orientation from $E(2)$. 

\begin{remark}
 We could define $X_p$ to be $E(2)_{\infty;\infty;2p+1}$, and take $I_{\cc_q}$ to be the Seiberg-Witten invariants corresponding to $2q[F_p]$, where $[F_p]$ is Poincar\'e dual to the multiple fiber in $X_p$. There are suitable family versions of this invariant~\cite{konno:classes}. In a forthcoming paper we will show how to compute the family Seiberg-Witten invariants to make this whole story work. This paper is not yet available, so we will describe a more complicated construction that is based on our computation~\cite{auckly-ruberman:diffym} of the relevant family Donaldson invariants.  

 While the family version will have to wait for the computation of the family Seiberg-Witten invariants, we note that these manifolds are sufficient to establish the existence of an infinite collection of exotic codimension one embeddings as described in section~\ref{S:ss} via Donaldson invariants. Indeed, by \cite[Theorem 3.1]{Fintushel-Stern:Simple type} and~\cite[Theorem 3.1]{gompf-mrowka},
\[ 
q_{X_p,\cc,\beta} = \sum_{j=-p}^p (-1)^{\frac12(\cc^2+2jc[f^{(3)}]+c[f^{(3)}])}. 
\]
\end{remark}

The construction of a collection of manifolds $X_p$ suitable for use with family Donaldson invariants generalizes $\xx = E(2)\cs\cptwobar$ to
\[
\xx^r  = \xx\cs_T E(2)^{(\cs_T)^{r+1}}\!\!\cs\cptwobar.
\]
The torus sum in this expression is done along the fibers labeled $T^{(1)}$ in each copy of $E(2)$. It is useful to notice that all of these are homeomorphic relative to $N^{(1)}$ to the manifold $\yy^r$ obtained by performing a cork twist on the cork in the zeroth copy of $\xx$. Finally, we define the manifolds $\xx^r_p$ by replacing the $p^{th}$ copy of $E(2)$ in $\xx^r$ with $E(2)_{\infty;2;2}$. We will take $X_p = \xx^r_p$ for our proof of Theorem~\ref{T:B}. We will take $\Sigma_1$ and $\Sigma_2$ to be parallel copies of the fiber $T^{(1)}$ in the first nucleus.

We now discuss the invariants that we use to distinguish the various families that we construct.  The same circle of ideas provides invariants for families of diffeomorphisms and for families of embeddings. We start with a review of the simplest sort of Donaldson invariants and then move on the family version, for the applications to families of diffeomorphisms and embeddings.

For a $4$-manifold $X$ with $b^2_+(X) > 1$, set 
\[
\ell_X  =  -\frac34 \left(\chi(X)+\sigma(X)\right), 
\]
and define $C_X$ to be the set of $\eta \in H^2(X;\zz_2)$ such that 
\begin{enumerate}
\item[(i)] $\eta$ lifts to an integral class $c \in H^2(X)$
\item[(ii)] For any integral lift $c$ of $\eta$ we have $c \cdot c \equiv \ell_X \pmod{4}$
\item[(iii)] $\eta \neq 0$.
\end{enumerate}
Further, we let $\Chat_X$ be the set of integral lifts $\cc$ of elements in $C_X$ modulo the equivalence $\cc \sim \cc'$ if $\cc \equiv \cc' \pmod2$ and $(\frac12(\cc-\cc'))^2  \equiv 0 \pmod2$. The set $C_X$ parameterizes a set of potential second Stiefel-Whitney classes of SO$(3)$ bundles over $X$. Adding the requirement that the virtual dimension of the ASD moduli space be zero then ensures that each element of $C_X$ corresponds to a unique bundle. The $\Chat_X$ refines the choice of bundle by an equivalence class of $\spinc$-lift up to an equivalence relation. In addition, we take a homology orientation, i.e., $\beta\in\Bigwedge^{top}(H^0\oplus H^2_+(X))$. The choice of $\cc$ and $\beta$ determine a specific SO$(3)$ bundle and the data to define an orientation on the corresponding ASD moduli space. Furthermore, for generic metrics these moduli spaces are a finite set of oriented points and the degree zero Donaldson invariant $q_{\cc,\beta}(X)$ is just the signed count of these points. These are the invariants in item (1) of the recipe from section~\ref{s:recursion}

For future reference, we note a special case of Donaldson's blowup formula~\cite{donaldson:polynomial}. If $\pi\colon X\cs\cptwobar\to X$ is the natural projection and  $\cc\in\Chat_X$, then $q_{X\cs\cptwobar,\pi^*\cc} = q_{X,\cc}$.
We build many manifolds via submanifold sums, all of which are sums along tori, so we need to know how to compute invariants in this situation. Here is a specific fiber sum formula that works in this setting. Let $X$ and $Y$ be smooth $4$-manifolds, each with a preferred embedded torus of square zero, each denoted $T$. Let $\cc_X\in\Chat_X$ and  $\cc_Y\in\Chat_Y$ satisfy $\cc_X[T] \equiv \cc_Y[T] \equiv 1\pmod{2}$. Further suppose that the restriction of  $\cc_{XY}\in\Chat_{X\cs_TY}$ to $X\setminus T$ (respectively $Y\setminus T$) agrees with the restriction of $\cc_X$ (respectively $\cc_Y$). Then by \cite[Theorem 2.2]{morgan-szabo}
\[
q_{{X\cs_TY},\cc_{XY}} = q_{X,{\cc_X}}q_{Y,{\cc_Y}}.
\]

We also need a computation of the Donaldson invariants of the $\xx^r_p$. We identify the cohomology groups $H^*(\xx^r)$ and $H^*(\yy^r)$ via a homeomorphism $\psi$. To specify the cohomology classes that we need for the Donaldson invariant calculation, we need a little notation. In the $i^{th}$ copy of $E(2)$, let the homology classes of the sections (lying in the $3$ nuclei) be denoted by $\sigma^{(i,a)}$ and the homology classes of the fibers be denoted by $f^{(i,a)}$.
Let $\Upsilon^\xx\colon\xx\setminus T^{(1)} \to \xx^r$ and $\Upsilon^{(i)}\colon E(2)^{(i)}\setminus f^{(i,1)} \to \xx^r$ be the inclusion maps. Let $\cc_q^\xx$ be cohomology classes in $H^2(\xx^r)$ so that $(\Upsilon^\xx)^*\cc_q^\xx$ is the image of the Poincar\'e dual to the section of $\xx$ in $H^2(\xx\setminus F)$, $(\Upsilon^{(i)})^*\cc_q^\xx$ is the image of the Poincar\'e dual to $\sigma^{(i,1)}+\sigma^{(i,2)}+\sigma^{(i,3)}+f^{(i,1)}-3f^{(i,2)}-2f^{(i,3)}$ in $H^2(E(2)^{(i)}\setminus f^{(i,1)})$ for $i\neq q$ and is the image of the Poincar\'e dual to $\sigma^{(i,1)}$ for $i=q$. Set $\cc_q = \psi^*\cc_q^\xx$. It is an elementary calculation to verify that all of the $\cc$-classes that we have defined live in the correct $\Chat$ sets. Furthermore, the Donaldson invariants satisfy the assumption that $(q_{\xx^r_p,\cc_q})$ is $4(-1)^r$ times the identity matrix \cite{auckly-ruberman:diffym}.

In~\cite{auckly-ruberman:diffym} the degree zero Donaldson invariant is extended to an invariant for families of diffeomorphisms.  Assume $b^2_+(Z)$ is at least $k+3$ and that $b^2_+(Z)$ and $k$ have the same parity. Set 
\[\ell_Z^k = -\frac34 \left(\chi(Z)+\sigma(Z)\right) + \frac12(k + 1) = -\frac12(3b_2^+(Z) - k + 2)  \]
and define $C_Z^k$ to be the set of $\eta \in H^2(Z;\zz_2)$ such that 
\begin{enumerate}
\item[(i)] $\eta$ lifts to an integral class $c \in H^2(Z)$
\item[(ii)] For any integral lift $c$ of $\eta$ we have $c \cdot c \equiv \ell_{Z}^k \pmod{4}$
\item[(iii)] $\eta \neq 0$.
\end{enumerate}
Further, we let $\Chat_{Z}^k$ be the set of integral lifts $\cc$ of elements in $C_{Z}^k$ modulo the equivalence $\cc \sim \cc'$ if $\cc \equiv \cc' \pmod2$ and $(\frac12(\cc-\cc'))^2  \equiv 0 \pmod2$.

Let $\alpha =\sum a_i\sigma_i$ be a smooth singular cycle representing a class in $H_k(\diff(Z))$. For a generic metric $g_*$ on $Z$ one gets a smooth singular cycle in $C_k(\text{Met}(Z))$ given by $\alpha^*g_* =\sum a_i\sigma_i^*g_*$. As $H_k(\text{Met}(Z)) = 0$, this is the boundary of a $(k+1)$-chain $\beta = \sum b_j\tau_j$. Informally, we have an invariant
\[
D^{H_k\diff}_{\cc,\beta}\colon  H^k(\diff(Z))\to \Z,
\]
given by
\[D^{H_k\diff}_{\cc,\beta}(\alpha) = \sum b_j\#\mathcal{M}_{\cc,\beta}(\{\tau_j\}).\]
To make the definition of this invariant rigorous, a variation of smooth singular chains that incorporates suitable transversality may be used. This is the approach described in~\cite{auckly-ruberman:diffym}. 
Composition with the Hurewicz homomorphism $\pi_k(\diff(Z))\to H_k(\diff(Z))$ produces an invariant for the homotopy groups of the diffeomorphism group. 

In~\cite{konno:classes} Konno defined a version of the degree zero Donaldson invariants for families of manifolds. His definition yields a cohomology class on $B\diff(Z)$ (or precisely on the classifying space for a subgroup of $\diff(Z)$ that preserves bundle data and homology orientations). He also describes it as a characteristic class for bundles of $4$-manifolds.  Konno's invariant is based on a CW representation of homology, and uses a family version of fundamental cycles. Given $\cc\in \Chat^k_Z$ and a homology orientation $\beta$ on $Z$, the version of his invariant that we use for a $(k+1)$-dimensional family $Z\hookrightarrow\mathcal{Z}\to B$ is denoted by $\mathbb{D}_{\cc,\beta}(\mathcal{Z})$. Essentially, it assigns to each $(k+1)$-cell a count of
ASD connections in a parameterized moduli space. Konno's invariants play the role of $\mathbb{I}_{\cc_q}(\mathcal{Z}[\![J]\!])$ from section~\ref{s:recursion}. The 
invariants playing the role of $I^{H_{k+1}\emb}_{\cc_q}$ in item (3) of the recipe from section~\ref{s:recursion} are given by
\[
D^{H_{k+1}\emb}_{\cc,\beta}(J) = \mathbb{D}_{\cc,\beta}(\mathcal{Z}[\![J]\!]).
\]
The following proposition establishes one relation between the family and diffeomorphism invariants.
\begin{proposition}\label{fam-diff}
    For $\alpha\in\pi_k(\diff(Z))$ one has
    \[
D^{H_k\diff}_{\cc,\beta}(\alpha) = \mathbb{D}_{\cc,\beta}(\mathcal{Z}[\alpha]).
    \]
\end{proposition}
\begin{remark}
    Intuitively, this proposition is clear. In simple situations one may represent a singular chain by a map from a CW complex into the space. Given transversality, the virtual neighborhood is trivial and the two definitions will agree. Drouin includes a proof of this proposition in~\cite{drouin:embeddings}.
\end{remark}

To compute the invariants $D^{H_{k+1}\emb}_{\cc,\beta}(J)$, we combine Proposition~\ref{fam-diff} with the homomorphism properties of $D^{H_k\diff}$ and following result from \cite{auckly-ruberman:diffym}.
\begin{theorem}[Anti-holomorphic blow-up]\label{anti-hol} Let $\beta\colon S^k\to\diff(Z,D)$ be a family so that $\beta\cs 1_{\mm}$ smoothly contracts via a contraction $F$.
We then have
\[
\lvert D_{\cc+\nnu}^{\pi_{k+1}}(F(1_Z\cs \bff)(F)^{-1})\rvert = 2 \lvert D_\cc^{\pi_k}(\beta)\rvert.
\] 
\end{theorem}
Putting everything together this implies that there are constants $\gamma_p'$ so that
\[
D^{H_{k+1}\emb}_{\cc_q,\beta}(J_p^{k+1}) = \gamma_p'\delta_{pq}.
\]

\section{Proof of Theorem~\ref{T:B}}\label{fam-pr}
Sections~\ref{S:inv} and~\ref{S:blocks} nearly complete the proof of Theorem~\ref{T:B}. The only missing piece is that one may take one target manifold $Z$ for which there are large rank summands in the kernels of the homology and homotopy groups for a range of indices. 

As we have defined them, the manifolds $Z_{k}'$ grow as $k$ grows. Picking an upper bound $n$ we can pick a new family of manifolds $V_k$ so that 
\[
Z_{k}'\cs_{\Sigma_1}V_k \cong Z_n''
\]
independent of $k$ for $k\le n$. Provided this only changes the invariants by a non-zero factor, the same argument will work. The reason this construction only addresses homotopy or homology groups in degree of a fixed parity is that the invariant $D^{H_{k+1}\emb}_{-,\beta}$ can only be non-zero when $b^2_+(Z) \equiv k+1 \ (\text{mod} \ 2)$. This argument is written out in more detail in~\cite{auckly-ruberman:diffym}.\hfill\qedsymbol
\begin{remark}
    The reason we just obtain finite-rank summands is that $\Chat^{k}_Z$ is finite. The Seiberg-Witten version of the family invariants may be parameterized by the (infinite) collection of $\spinc$ structures. Thus, in a forthcoming paper
    we describe how to compute the relevant Seiberg-Witten invariants that establish the existence of infinite-rank summands of homotopy classes of embeddings into a fixed $4$-manifold. It is not unreasonable to imagine that a family version of the Bauer-Furuta invariants could establish that the kernels of all of the homotopy and homology groups in a range contain infinite-rank summands. 
\end{remark}

To construct exotic embeddings of a $3$-manifold $Y$ we used an embedding of $Y$ into a $4$-manifold $V$ with non-trivial Donaldson invariants. This leads to the following question (see also~\cite{etnyre-min-mukherjee:boundaries}).

\begin{question}
    Does every $3$-manifold embed into some $4$-manifold with non-trivial Donaldson invariants? Non-trivial Seiberg-Witten invariants? Non-trivial Bauer-Furuta invariants? As the referee points out, we should require that the $3$-manifold separates to avoid a trivial answer such as $Y \hookrightarrow (S^1 \times Y) \cs X$ where $X$ has a non-trivial gauge-theoretic invariant. 
 \end{question}
It is known~\cite{daemi-lidman-miller:nosymp} that there are $3$-manifolds that do not embed smoothly in any symplectic $4$-manifold, blocking an easy route to a positive answer to this question.

\section{Families of submanifolds}
The previous section gave explicit families of \emph{embeddings} of submanifolds; by definition this includes the parameterization. Our goal is to show that in some cases, these families of embeddings give rise to families of \emph{submanifolds}. This means that we need to understand any possible dependence of the invariants from the choice of parameterization, or in other words, of automorphisms of the submanifold. In the following discussion, we formalize the family version of automorphisms of the submanifold.  Let $Y\hookrightarrow \mathcal{Y}\to \Xi$ be a family  of manifolds and $Z\hookrightarrow \mathcal{Z}\to \Xi$ be a family of manifolds with $\dim(Z) \geq \dim(Y)$. 
Let $\text{Aut}(\mathcal{Y})$ be the group of bundle automorphisms of $\mathcal{Y}$. 
    \begin{definition}
        A family of embeddings is a bundle map $J\colon \mathcal{Y}\to\mathcal{Z}$. Define an equivalence relation on families of embeddings by $J\sim J'$ if there is $\xi\in \text{Aut}(\mathcal{Y})$ so that $J\xi = J'$. A family of submanifolds is an equivalence class of a family of embeddings.
    \end{definition}

    \begin{remark}
        We denote the space of family embeddings by $\mathcal{E}\text{mb}(\mathcal{Y},\mathcal{Z})$. In this case the family of submanifolds with given bundles of manifolds is identified with 
        \[\mathcal{E}\text{mb}(\mathcal{Y},\mathcal{Z})/\text{Aut}(\mathcal{Y}).\] We call a family of embeddings representing a given family of submanifolds a \emph{marking}.
    \end{remark}

The first invariant of a family of submanifolds is the underlying bundle $\mathcal{Y}$. Such bundles are classified by homotopy classes of maps to $\text{BDiff}(Y)$. The identification $\Omega\text{BDiff}(Y) \simeq \diff(Y)$ provides a good understanding of the homotopy type of $\text{BDiff}(Y)$. For example
if $Y$ is a hyperbolic $3$-manifold, $\diff(Y)$ is homotopy equivalent to $\text{Isom}(Y)$, so $\text{BDiff}(Y)$ is a $K(\text{Isom}(Y),1)$.  It would be appropriate to restrict to orientation-preserving diffeomorphisms and isometries to study oriented families of submanifolds.

Now turn to the question of distinguishing families of submanifolds with the same underlying bundle of spaces. For simplicity assume this bundle is trivial, so $\Xi\times Y$. The basic invariant for 
families of embeddings was to consider a family of spaces constructed via family submanifold sum, $\mathcal{Z}\cs_J\mathcal{V}$. The trouble with this approach for families of submanifolds is that the isomorphism class of the family of spaces may depend upon the representative of the family of submanifolds. One solution to this issue is to take the families associated to all markings of the submanifold. Since markings that differ by elements of $\text{Aut}(\mathcal{Y})_0$ give rise to the same families of maps, a better option is to define a homotopy marking as an equivalence class of marking up to $\text{Aut}(\mathcal{Y})_0$. The natural invariant of a family of submanifolds is the multiset (set with multiplicities) of all families of spaces obtained by taking the submanifold sum over all homotopy markings.  If that multiset is finite, then one can simplify further by taking the sum of invariants over the multiset.

A similar idea works when considering higher-dimensional families of submanifolds The higher homotopy groups are groups, so it is natural to look for an invariant in the form of a group homomorphism. If the diffeomorphism group of $Y$ has finitely many components (eg if $Y$ is hyperbolic) then sum of the invariants over all markings {will} 
give an integer-valued homomorphism from the homology groups of the space of submanifolds.

\subsection{Examples}\label{S:examples}

We now consider some examples to illustrate the difference between detecting families of embeddings and families of submanifolds. In this section we will return to the choice of the manifolds $X_p$ as the $(2p+1)$-log transforms of $E(2)$. This will simplify the discussion, and provide equally powerful results once combined with the forthcoming work on the Seiberg-Witten invariants. The first example illustrates that the invariants for $\pi_0(\emb)$ defined in section~\ref{S:ss} can in fact depend on the marking. We address this in two ways: one is by considering the multiset of invariants, and the other is by taking the sum over all markings (when the set of markings is finite). The second example is one where the diffeomorphism group of $Y$ is contractible, so there is only one marking. It appears as the second example because it is constructed by modifying the first example.  In fact, we describe a way to replace a given embedded submanifold by one with a contractible diffeomorphism group. 
\begin{example}\label{E:selman}
The first class of examples is constructed with the submanifold $Y= \partial W$ where $W$ is the Akbulut-Mazur cork. To start with, we will consider the dependence of our $\pi_0(\emb)$ invariants on the parameterization. 

To understand the diffeomorphism group of $Y$, we use the surgery picture in Figure~\ref{AM} in which all of the symmetries are visible.
\begin{figure}[ht]
\labellist
\small\hair 2pt
\pinlabel {$\tau$} [ ] at 42 6
\pinlabel{$\sigma$} [ ] at 95 55
\pinlabel {$+1$} [ ] at 5 95
\endlabellist
\centering
\includegraphics[scale = 1]{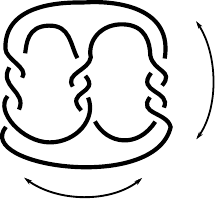}
  \caption{The boundary of the Akbulut-Mazur manifold}\label{AM}.
\end{figure}
Combining results about surgery on Montesinos knots shows that $Y$ is hyperbolic~\cite{meier:small,menasco-thistlethwaite:alternating,wu:montesinos}.
A calculation with SnapPy shows further that 
\[
\diff(Y) \simeq \text{Isom}(Y) =\langle \sigma, \tau \,|\, \sigma^2 = \tau^2 = [\sigma,\tau] = 1\rangle
\]
where $\sigma$ and $\tau$ are exhibited in Figure~\ref{AM}. This calculation, because it takes place on a closed manifold, is not completely rigorous. Forthcoming work of the second author with N. Hoffman gives a method to transform such calculations into a calculation of the symmetry group of an associated link, which can be done rigorously.

The diffeomorphism $\sigma$ extends over $W$, while the diffeomorphism $\tau$ is the cork twist. 
By~\cite{akbulut:contractible}, the manifold $W$ embeds into $V = E(2)\cs \cptwobar$ so that 
\[
\left(E(2)\cs\cptwobar\setminus\text{int}(W)\right)\cup_\tau W \cong 3\cptwo\cs 20\cptwobar
\]
from which it follows that $\tau$ does not extend over $W$.  
Moreover, it is known~\cite[Figure 9.4]{gompf-stipsicz:book} that one can take this embedding to be disjoint from a fiber. 

Set $\overline{V_1} = W$ and $U = \sss$. 
Label the three nuclei in the original $E(2)$ as $N^{(0,a)}$ for $a = 1, 2, 3$, and label the nuclei in the transformed copy of $E(2)$ by  $N^{(1,a)}$ for $a = 1, 2, 3$. Denote the section classes by $\sigma^{(i,a)}$ and the fiber classes by $f^{(i,a)}$. After the log transforms to obtain $E(2)_{\infty;\infty;2p+1}$ we will have $N^{(1,3)}_{2p+1}$. Figure~\ref{N2_2} displays the result of a $(2p+1)$-log transform on  a nucleus. Let $f^{(3)}$ denote the $0$-framed handle and $s^{(3)}$ denote the $(-8p^2-6p-2)$-framed handle
in the displayed figure of $N^{(1,3)}_{2p+1}$.
\begin{figure}[ht]
\labellist
\small\hair 2pt
\pinlabel {$-2$} [ ] at 23 87
\pinlabel{$2p\!+\!1$} [ ] at 82 27
\pinlabel {$-8p^2-6p-2$} [ ] at 164 27
\pinlabel{$0$} [ ] at 218 50
\endlabellist
\centering
\includegraphics[scale = 1]{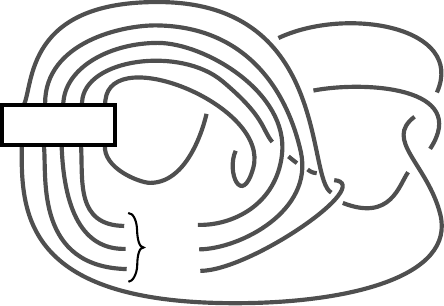}
  \caption{The Gompf nucleus after log transform $N^{(1,3)}_{2p+1}$}\label{N2_2}.
\end{figure}
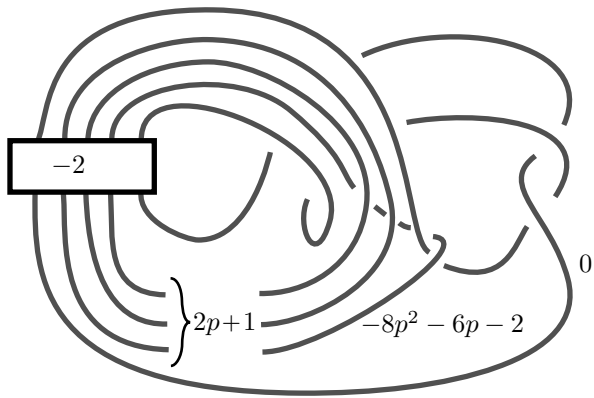

By the work of Freedman, there is a homeomorphism 
rel boundary
$\zeta\colon  N^{(1,3)}_{2p+1}  \to  N^{(1,3)}$ that matches the {multiple fiber class $f^{(3)}$ in the transformed nucleus} with the fiber class in the nucleus on the other side. Furthermore, 
the class $s^{(3)}$ is sent to $\sigma^{(1,3)} - (4p^2+3p)f^{(1,3)}$.
Notice that after a log transform, the generic class is no longer primitive. Indeed the multiple fiber is primitive, and the generic fiber is a multiple of the multiple fiber. In this case the generic fiber is in the homology class of $(2p+1)f^{(3)}$ and $f^{(3)}$ is primitive. See \cite[Remark 3.3.5]{gompf-stipsicz:book}

In this case we have $\psi_p\colon  E(2)_{\infty;\infty;2p+1} \to E(2)$ extending $\zeta$ by the identity. We also have a diffeomorphism $\varphi_p\colon E(2)_{\infty;\infty;2p+1}\cs_T\widehat{V} \to E(2)\cs_T\widehat{V}$ as in the introduction. We then have the embeddings $j_p\colon Y\to  E(2)\cs_T\widehat{V}$ with $j_p = \varphi_p(\psi_p\cs_T1_{\widehat{V}})^{-1}j$, and the submanifolds $j_p(Y)$. We also have a model embedding $k\colon Y\to V$ that we can use to form submanifold sums to distinguish the submanifolds in $X$.

If we now consider the submanifolds $j_p(Y)$  we can still construct the submanifold sums, but the answer depends on the choice of marking, so we compare the manifolds obtained as submanifold sum using the markings $\{j_p,j_p\sigma,j_p\tau,j_p\tau\sigma\}$ for various $p$. The preferred components of the submanifold sums for the markings associated to $j_p(Y)$ are:
\[
\{E(2)_{\infty;\infty;2p+1} \cs_T V, E(2)_{\infty;\infty;2p+1} \cs_T V, 7\cptwo\cs 40\cptwobar, 7\cptwo\cs 40\cptwobar \}.
\]
In all cases the discarded components are $4\cptwo\cs 21\cptwobar$.

Let $\cc$ be the element of $\Chat_{E(2)\cs_TV}$ that restricts to the Poincar\'e duals of $\sigma^{(0,1)}$ on the $E(2)$ summands (recall that $V = E(2)\cs\cptwobar$). Let $\beta$ be the homology orientation induced by the submanifold sum. Applying the degree zero Donaldson invariant $q_{-,\cc,\beta}$ to this set of manifolds gives $\{2p+1,2p+1,0,0\}$. The dependence of the multiset of Donaldson invariants on $p$ demonstrates that the submanifolds for different values of $p$ are not smoothly isotopic. As remarked above, since the number of homotopy markings is finite in this case, we could alternatively sum the invariants over all homotopy markings. 
 
The general construction described in section~\ref{fam-pr} now gives spherical families $J^{k+1}_p$ of embeddings of $Y$. These in turn generate the families of submanifolds  given by $J^{k+1}_p(Y)$. As in the case of $\pi_0$ the multiset of submanifold sums taken over all markings will have four elements. Applying the degree zero family Donaldson invariants $\mathbb{D}_{\cc_q,\beta}$ will result in the multiset of values $\{\gamma_p\delta_{pq}, \gamma_p\delta_{pq}, 0,0\}$ for exactly the same reason. The first two elements arise from $J^{k+1}_p(Y)$ and $J^{k+1}_p\sigma(Y)$. The last two zero values arise from $J^{k+1}_p\tau(Y)$ and $J^{k+1}_p\tau\sigma(Y)$. The sum over all homotopy markings then gives non-trivial homomorphisms to $\Z$; using different choices for $\cc_q$ gives a surjection from the homology group of the space of submanifolds to free abelian groups of arbitrarily high (finite) rank.
\end{example}
    The easiest case to analyze is when the diffeomorphism group of $Y$ is contractible. This occurs when $Y$ is hyperbolic with no symmetries (i.e. is \emph{asymmetric}.) We start with a construction of such a $Y$ as a modification of the Akbulut cork. The construction works in much greater generality, but we will content ourselves with a particular case.  Recall that cobordism $P$ from a manifold $M$ to a manifold $N$ is invertible if there is a cobordism $Q$ from $N$ to $M$ with $P \cup_N Q \cong M \times I$.  There is a similar notion of invertible concordance for knots, links, and tangles, where in the latter setting one requires a fixed product structure along the boundary. The relevance of invertible cobordisms to embedding problems is the following obvious lemma.
\begin{lemma}
    A co-oriented embedding $j: N^{m-1} \hookrightarrow X^m$ and an invertible cobordism from $M$ to $N$ gives rise to a co-oriented embedding $M  \hookrightarrow X$.
\end{lemma}    

\begin{example}
This example exhibits a method for converting a $3$-manifold $N$ into one with a trivial symmetry group such that there is an invertible cobordism from $N$ to the new manifold. The idea is to insert a tangle with trivial symmetry group into a surgery diagram for $N$. The method seems to work in some generality and can easily be applied to particular examples.
\begin{lemma}\label{L:pkta}
   Let $Y$ be the Akbulut-Mazur cork, given as $(+1)$ surgery on the pretzel knot $P(-3,3,-3)$ as drawn in Figure~\ref{AM}.
   Then there is an invertible homology cobordism from $Y$ to an asymmetric hyperbolic homology sphere. More precisely, there are infinitely many distinct asymmetric hyperbolic homology spheres $Y_n$, and invertible cobordisms from $Y$ to $Y_n$.
\end{lemma}
\begin{proof}
The manifold $Y_n$ is given by the following surgery diagram on the two component link $L=(K,A)$.  A calculation with SnapPy~\cite{SnapPy} shows that $L$ is hyperbolic with trivial symmetry group and moreover the manifold obtained by Dehn surgery on $K$ with coefficient $-1$ has the same properties.  The calculations are done via interval arithmetic and are therefore rigorous.  A standard argument~\cite{hodgson-weeks:symmetries} using Thurston's Dehn surgery theorem shows that for sufficiently large $n$, the manifolds $Y_n$ are asymmetric, hyperbolic, and distinct.  
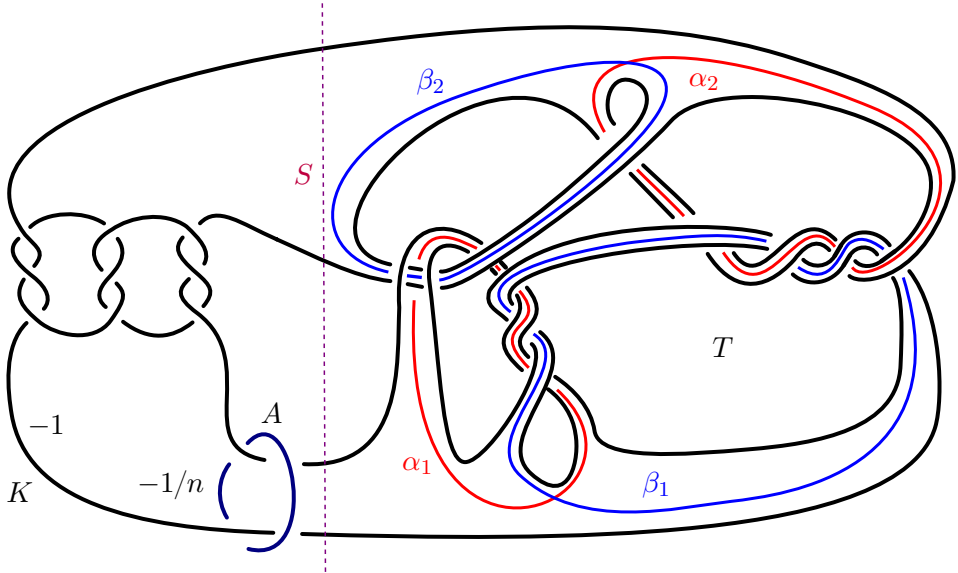
\begin{figure}[ht]
\labellist
\pinlabel {$P(-3,3,-3)$} [ ] at 30 545
\pinlabel {$T$} [ ] at 270 85
\pinlabel{$-1/n$} [ ] at 62 33
\pinlabel{$-1$} [ ] at 15 55
\pinlabel {$A$} [ ] at 100 60
\pinlabel {$K$} [ ] at 5 30
\pinlabel {$\Purple{S}$} [ ] at 112 150
\pinlabel {$\Red{\alpha_1}$} [ ] at 155 40
\pinlabel {$\Blue{\beta_1}$} [ ] at 245 33
\pinlabel {$\Red{\alpha_2}$} [ ] at 263 185
\pinlabel {$\Blue{\beta_2}$} [ ] at 160 185
\endlabellist
\centering
\includegraphics[scale=1]{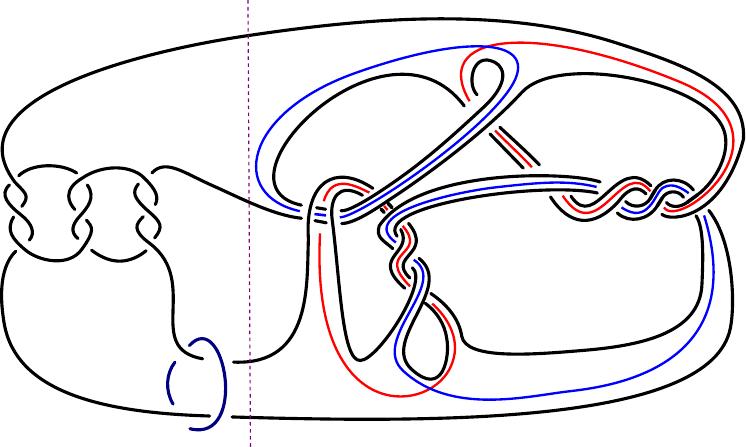}
  \caption{The $P(-3,3,-3)$ knot with DS tangle $T$ and augmentation $A$}\label{F:PDR}.
\end{figure}

It remains to construct the invertible cobordism; the construction is similar to the use of the `KT grabber' in~\cite{bleiler:prime}. Note first that doing $-1/n$ surgery on $A$ turns $K$ into a knot $K_n$ that has $n$ full twists in the region where it passes through a disk bounded by $K$. We will show that for all $n$, the pretzel knot $P(-3,3,-3)$ is invertibly concordant to $K_n$.  Note that the knot $K_n$ splits along the Conway sphere indicated by the vertical dotted line into a sum of two tangles. On the left is a tangle whose closure is the pretzel knot $P(-3,3,-3)$; let us call the portion on the right the tangle $T$. If one replaces $T$ by a trivial tangle, the resulting knot is the pretzel knot, because the extra twists come undone. 

There is an invertible cobordism from the trivial tangle to $T$, constructed as in the proof of~\cite[Theorem 2.6]{ruberman:seifert}. The main points in verifying the invertibility can be seen in Figure~\ref{F:PDR}, which is a variation of~\cite[Figure 1]{ruberman:seifert}. (Additional twists have been added to the bands to break the symmetries of the older picture.) The tangle $T$ has two components, each of which bounds a genus $1$ relative Seifert surface (with part of the boundary along $T$ and the other part along the sphere $S$) and that those surfaces are disjoint.  Observe that the Seifert form on homology of the lower surface has a hyperbolic basis $\{\alpha_1, \beta_1\}$, with a similar basis $\{\alpha_2, \beta_2\}$ for the upper surface.

View $T$ as lying in $B^3 \times \{1/2\} \subset B^3 \times I$. If one does surgery on either $\{\alpha_1, \alpha_2\}$ or  $\{\beta_1, \beta_2\}$, both of those surfaces become disks. Since $\{\alpha_1, \alpha_2\}$ is a trivial link one can add $2$-handles in $B^3 \times [0,1/2]$ to do the surgery, and similarly add $2$-handles in $B^3 \times [1/2,1]$ to do surgery along $\{\beta_1, \beta_2\}$.  The union of the traces of the two surgeries yields two copies of $D^2 \times I$ meeting $S \times I$ in a pair of arcs $\times I$.  

Gluing this cobordism to the product cobordism over the left side of the picture gives an invertible cobordism between the pretzel knot and $K_n$.
\end{proof}
\end{example}

\begin{example}
    Recall that the Akbulut-Mazur cork in $E(2)\cs\cptwobar$ is disjoint from the first Gompf nucleus. As $-\partial N(2) = \Sigma(2,3,11)$, we see that the Brieskorn sphere $Y = \Sigma(2,3,11)$ embeds into $E(2)\cs\cptwobar$ separating the torus used for fiber sums from the Akbulut-Mazur cork. By the same argument as in the application of Theorem~\ref{T:B} in Example~\ref{E:selman}, this embedding of $Y$ gives rise to $r$ distinct isotopy classes of embeddings $j_p$ distinguished Donaldson invariants $\mathbb{D}_{\cc_q,\beta}$. Likewise, there are rank $r$ large summands in the homotopy and homology groups of $\emb(Y,Z^k_r)$.  
   
Now consider the extension of these results to spaces of submanifolds; the key issue again is to understand the diffeomorphism group of $Y$. It is known~\cite{boileau-otal:small} that the mapping class group of $Y$ is $\Z_2$.  Moreover, the generator, say $\tau$ is represented by complex conjugation in a standard description of $Y$ as a link of an isolated singularity.  By \cite{MS} each component of the group $\diff(\Sigma(2,3,11))$ has the homotopy type of a circle.  

We first claim that the submanifolds $j_p(Y)$ are not isotopic. To show this, we note that Gompf showed~\cite[Lemma 3.7]{gompf:nuclei} that $\tau$ extends to a self-diffeomorphism of $N(2)$. It follows that $j_p$ and $j_p\tau$ have the same Donaldson invariants, so the multisets of invariants distinguish the submanifolds $j_p(Y)$. With regard to the higher-dimensional families, the long exact sequence
\[
\cdots \to \pi_n(\diff(Y)) \to \pi_n(\emb(Y,Z_{k})) \to \pi_n(\sub(Y,Z_{k})) \to \pi_{n-1}(\diff(Y)) \to \cdots
\]
shows that the families $J_p^{j+1} \in \pi_{j+1}(\emb(Y,Z^k_r))$ are independent.
\end{example}

\def\cprime{$'$}
\providecommand{\bysame}{\leavevmode\hbox to3em{\hrulefill}\thinspace}

\end{document}